\documentclass[letterpaper]{amsart}

\usepackage[left=1.455in,right=1.455in,top=1in,bottom=1in]{geometry}

\usepackage[dvipsnames]{xcolor}

\usepackage{macros}

\begin{document}

\begin{abstract}
  The standard period-index conjecture for Brauer groups of $p$-adic
  surfaces $S$ predicts that $\ind(\alpha)|\per(\alpha)^3$ for every
  $\alpha\in\Br(\QQ_p(S))$.  Using Gabber's theory of
  {prime-to-$\ell$} alterations and the deformation theory of twisted
  sheaves, we prove that $\ind(\alpha)|\per(\alpha)^4$ for $\alpha$ of
  period prime to $6p$, giving the first uniform period-index bounds
  over such fields.
\end{abstract}

\maketitle

\tableofcontents

\vspace*{-2em}
\section{Introduction}
\numberwithin{theorem}{section}

The purpose of this paper is to prove the following result concerning
the period-index problem for the Brauer group.

\begin{theorem}\label{thm:main}
Let $R$ be an excellent henselian discrete valuation ring with residue
field $k$ of characteristic $p\geq 0$ and with fraction field $K$. Suppose
$k$ is semi-finite or separably closed.  Let $L$ be an extension of
$K$ of transcendence degree $2$, and let $\alpha\in\Br(L)$ be a Brauer
class.  If $\alpha$ has period prime to $p$, then
$$
\ind(\alpha) \mid \per(\alpha)^5.
$$
If $\alpha$ has period prime to $6p$, then
$$
\ind(\alpha) \mid \per(\alpha)^4.
$$
\end{theorem}

Recall from~\cite{lieblich-tr2} that $k$ is a semi-finite field if it
is perfect and if for every prime $\ell$, the maximal prime-to-$\ell$
extension of $k$ is pseudo-algebraically closed with Galois group
$\ZZ_\ell$. Finite fields and pseudo-finite fields are semi-finite. As
a special case, we obtain the following result.

\begin{corollary}
Let $S$ be a geometrically integral surface over a $p$-adic field
$K$. If ${\alpha\in\Br(K(S))}$ has period relatively prime to $6p$,
then
$$
\ind(\alpha) \mid \per(\alpha)^4.
$$
\end{corollary}

The period of a Brauer class $\alpha$ is its order in the Brauer
group and its index is the degree of a division
algebra in the Brauer class. The period divides the index and both numbers have the same
prime factors.  Results bounding the index in terms of the period have
motivated many of the developments in the theory of the Brauer group
since the beginning of the 20th century.  See \cite{ketura}*{Section~4} for a
survey of results of this type.

For local and global fields, the index equals the period by Albert,
Brauer, Hasse, and Noether
(see~\cite{gille-szamuely}*{Remark~6.5.6}). For a finitely generated
field of transcendence degree 2 over an algebraically closed field,
the index equals the period by de Jong~\cite{dejong} (see also
\cite{lieblich}). More generally, Artin conjectured that the index
equals the period for every $C_2$ field, and he proved this for Brauer
classes of period a power of 2 or 3, see \cite{artin:Brauer-Severi}.
For a field of transcendence degree~1 over a local field, the index
divides the square of the period by Saltman~\cite{saltman-padic} for
Brauer classes of period prime to the characteristic and Parimala and
Suresh \cite{parimala_suresh:period-index} in general. Analogous
results for fields of transcendence degree 1 over higher local fields
are established in \cite{lieblich:arithmetic} and subsequently in
\cite{HHK} by other methods. For fields of transcendence degree 2 over
a finite field, the index divides the square of the period by
\cite{lieblich-tr2}.

Such results support the following conjecture (see
\cite{colliot:bourbaki05}*{Section~2.4}).

\begin{conjecture}[Period-index conjecture]
\label{conjecture}
Let $k$ be an algebraically closed, $C_1$, or $p$-adic field, and set
$e=0,1,2$ accordingly.  Let $K$ be a field of transcendence degree $n$
over $k$. For every $\alpha \in \Br(K)$, we have
$$
\ind(\alpha)|\per(\alpha)^{n-1+e}.
$$
\end{conjecture}

Based on this conjecture, we do not expect the period-index bound we
achieve in Theorem~\ref{thm:main} to be optimal.  However, this is the
first proof of a general period-index bound that is uniform in the
period for fields of transcendence degree 2 over a local field.  For
classes of period a power of 2, bounds on the $u$-invariant are known
to imply uniform bounds for the index in terms of the period; our
bounds are still better than what can be attained using known
$u$-invariant results for function fields over $p$-adic fields
\cite{leep}. There are also nonuniform period-index bounds for $C_i$
fields due to Matzri~\cite{matzri:symbol_length}.

Our approach follows a strategy inspired by
Saltman~\cite{saltman-padic}: split the ramification of the Brauer
class by a field extension of controlled degree and then use geometry
to study the unramified Brauer class on a regular proper model.  For
the former, we draw on, and expand upon, a development due to
Pirutka~\cite{pirutka} (see Section \ref{sec:ram}).  After splitting
the ramification and using Gabber's refined theory of
$\ell'$-alterations to reduce to a regular (quasi-semistable) model,
we reduce the proof of Theorem~\ref{thm:main} to the following
general result. Given an integral scheme $X$, we write $\kappa(X)$ for its
function field; given $\alpha\in\Hoh^2(\kappa(X),\m_n)$, we write
$\per(\alpha)$ and $\ind(\alpha)$ for the period and index of the
associated class in $\Br(\kappa(X))$.

\begin{theorem}\label{thm:1}
Let $R$ be an excellent henselian discrete valuation ring with residue
field~$k$ of characteristic $p\geq 0$ and with fraction field
$K$. Suppose that $X$ is a connected regular {$3$-dimensional} scheme,
flat and proper over $\Spec R$. Let $\alpha\in\Hoh^2(X,\m_n)$ where
$n$ is prime to~$p$. Assume that the Brauer class of $\alpha$ is
trivial on all proper closed subschemes of the reduced special fiber
$X_{0,\red}$ of dimension at most $1$. If
$\ind(\alpha|_{\kappa(X_i)})=\per(\alpha|_{\kappa(X_i)})$ for all
irreducible components $X_i$ of $X_{0,\red}$, then
$\ind(\alpha_{\kappa(X)})=\per(\alpha_{\kappa(X)})$.
\end{theorem}

Note that the hypothesis, that the Brauer class of $\alpha$ is trivial
on all proper subschemes of $X_{0,\red}$ of dimension at most $1$, is
automatically satisfied if $k$ is semi-finite or separably closed (see
Lemma~\ref{lem:curves}).

The proof of Theorem~\ref{thm:1} uses the deformation theory of
twisted sheaves to reduce the computation of the index of a Brauer
class on a regular model to the existence of twisted sheaves of a
certain rank on the reduced special fiber, which we can assume is a
strict normal crossings surface.  In the case when the special fiber
is smooth, this approach was carried out in
\cite{lieblich}*{Proposition~4.3.3.1}.  In the general case, we end up
proving a version of de Jong and Lieblich's period-index theorems for
strict normal crossings surfaces over separably closed and semi-finite
fields, respectively.

It is known that Saltman's theorem is the best possible for $p$-adic
curves. Indeed, examples were given by Jacob and Tignol in an appendix
to~\cite{saltman-padic} to this effect. Conjecture~\ref{conjecture}
predicts that for a surface over $\CC((t))$ one has
$\ind(\alpha)|\per(\alpha)^2$, while for a surface over a $p$-adic
field one has $\ind(\alpha)|\per(\alpha)^3$.  The nonoptimality of
our results is undoubtedly due to our overly generous splitting of
ramification. The approach taken in \cite{lieblich-tr2} improves these
kinds of bounds at the expense of a layer of stacky complexity.

\medskip
\paragraph{\bf Outline.}
In Section~\ref{sec:ram}, we generalize work of Pirutka~\cite{pirutka}
on splitting the ramification of Brauer classes.
Section~\ref{sec:alterations} considers Gabber's refined theory of
$\ell'$-alterations in the context of splitting ramification.
Sections~\ref{sec:snc} and \ref{sec:deformations}
discuss the existence and deformation theory of twisted sheaves on
proper models of the function field we consider.
Theorems~\ref{thm:main} and \ref{thm:1} are proved in
Section~\ref{sec:proof}. Starting in Section~\ref{sec:snc}, we freely
use the theory of twisted sheaves. An introduction to the use of
twisted sheaves to study questions about the Brauer group can be found
in \cite{lieblich-moduli} and \cite{lieblich}.

\medskip
\paragraph{\bf Notation.}
If $X$ is a scheme and $R$ is a commutative ring, we denote by
$\Hoh^i(X,\m_n)$, $\Hoh^i(R,\Gm)$, and so on the corresponding \'etale
cohomology groups, and by $\Br(X)$ and $\Br(R)$ the respective Brauer
groups of Azumaya algebras. Given a locally noetherian scheme $X$ and
a $\G$-gerbe $\pi : \Xscr \to X$ for some closed subgroup
$\G\hookrightarrow\G_m$, we will write $\Coh^{(1)}(\Xscr)$ for the
category of coherent $\Xscr$-twisted sheaves.  Similarly, given a
locally noetherian scheme $X$, we will write $\Coh(X)$ for the usual
categories of coherent sheaves on $X$.  When $F$ is an $\Xscr$-twisted
sheaf and $M$ is an $\Oscr_X$-module, for simplicity we write $F
\otimes M$ for the $\Xscr$-twisted sheaf $F \otimes_{\Oscr_\Xscr}
\pi^\ast M$.

\medskip
\paragraph{\bf Acknowledgments.}
We would like to thank the SQuaRE program of the American Institute of
Mathematics for its hospitality and support for this project. We also
thank the Banff International Research Station (BIRS) and the
Institute for Computational and Experimental Research in Mathematics
(ICERM) for wonderful working environments during the final stages of
preparation of this paper.  We thank Fran\c{c}ois Charles, Jean-Louis
Colliot-Th\'el\`ene, Ronen Mukamel, R. Parimala, Eryn Schultz, Lenny
Taelman, and Tony V\'arilly-Alvarado for helpful discussions.  We
especially thank Sam Payne and Dhruv Ranganathan for expert advice on
toroidal geometry.  Finally, we are very grateful to Minseon Shin for
several comments and corrections on the draft of this paper. Without
y'all, this paper would not be the same.  We are grateful to the
referees for their detailed comments on the paper.

During the course of this work, Lieblich was partially supported by
NSF CAREER grant DMS-1056129, NSF grant DMS-1600813, and a Simons Foundation Fellowship. Krashen was
partially supported by NSF Career grant DMS-1151252 and FRG grant
DMS-1463901. Ingalls was partially supported by an NSERC Discovery
grant. Auel was partially supported by NSA Young Investigator Grants
H98230-13-1-0291 and H98230-16-1-0321. Antieau was partially
supported by NSF Career grant DMS-1552766 and NSF grants DMS-1358832
and DMS-1461847.

\section{Splitting ramification}
\label{sec:ram}
\numberwithin{theorem}{subsection}

The results of this section are, for the most part, a generalization
and reworking of the results of Pirutka \cite{pirutka} (which
themselves live in a tradition of ramification-splitting results due
to Saltman~\cite{saltman-padic}).  We follow Pirutka's strategy with
minor modifications so that it works in mixed characteristic, and we
give a somewhat different argument on the existence of rational
functions whose roots split ramification.

Our ultimate goal in this section is to show that we can split all of
the ramification occurring in the Brauer classes of interest to us
with relatively small extensions.

Let $X$ be a regular noetherian integral scheme with function field
$F$.  Restriction to the generic point induces an injective map
$\Hoh^2(X,\Gm) \to \Hoh^2(F,\Gm)\iso\Br(F)$, see
\cite{grothendieck:Brauer_II}*{Proposition~1.8}.  For $A$ a
commutative ring, the canonical map $\Br(A) \to \Hoh^2(A,\Gm)_{\tors}$
is an isomorphism; see Hoobler~\cite{hoobler}.  More generally, if $X$
is a scheme admitting an ample invertible sheaf, then $\Br(X) \to
\Hoh^2(X,\Gm)_{\tors}$ is an isomorphism; see~\cite{dejong:gabber}.

\subsection{Ramification}
In this section we fix a ring $R$ and a field $F$ containing $R$.

\begin{definition}
  Fix a class $\alpha\in\Br(F)$.
  \begin{enumerate}
  \item The class $\alpha$ is \emph{unramified\/} at a discrete
  valuation ring $A$ of $F$ if $\alpha$ is in the image of the
  restriction map $\Br(A)\to\Br(F)$. Otherwise, we say that $\alpha$ is \df{ramified} at $A$.
  \item The class $\alpha$ is \emph{unramified over $R$\/} if $\alpha$
    is unramified at every discrete valuation ring $A$ of $F$ such
    that $R\subset A$.
    \item If $X$ is a regular noetherian integral scheme with function field $F$
    and $x \in X$ a point of codimension~1, we say that $\alpha$ is
    \emph{unramified at $x$\/} if $\alpha$ is unramified at the
    discrete valuation ring $\Oscr_{X,x}$ of $F$. When additionally $X=\Spec\, R$ and $x \in R$ is a nonzero divisor, we say that $\alpha$ is
    unramified at $x$ if it is unramified at the prime ideal $(x)$. In this
    circumstance, the Weil divisor consisting of the sum of all
    codimension 1 points of $X$ over which $\alpha$ is ramified is
    called the \df{(reduced) ramification divisor} of $\alpha$.
  \end{enumerate}
\end{definition}
\begin{remark}
Similarly, for a positive integer $\ell$ invertible in $F$, we can
consider the ramification of classes in $\Hoh^i(F,\m_{\ell}^{\tensor
j})$ at any discrete valuation ring $A$ of $F$ whose residue field
$\kappa$ has characteristic not dividing~$\ell$. In this case, $\alpha
\in \Hoh^i(F,\m_{n}^{\tensor j})$ is unramified if and only if $\alpha$
is contained in the kernel of the residue map $\Hoh^i(F,\m_{n}^{\tensor
j}) \to \Hoh^{i-1}(\kappa,\m_{n}^{\tensor j-1})$ defined in terms of
Galois cohomology, see \cite{colliot:santa_barbara}*{\S3.6}.
Important cases are $\Hoh^2(F,\m_\ell)$ and $\Hoh^i(F,\m_\ell^{\otimes
i})$, which correspond to Brauer classes of period $\ell$ and symbols
of length $i$ in Galois cohomology.
\end{remark}

\begin{lemma}\label{lem:unram-climbs}
  Suppose $L/F$ is a finite field extension. If $\alpha\in\Br(F)$ is
  unramified over $R$ then the restriction $\alpha_L\in\Br(L)$ is
  unramified over $R$.
\end{lemma}
\begin{proof}
  For any discrete valuation ring $A$ with fraction field $L$, the
  intersection $A\cap F$ is a discrete valuation ring with fraction
  field $F$, and if $\alpha$ is unramified at $A \cap F$, then the
  restriction $\alpha_L \in \Br(L)$ is unramified at $A$.
\end{proof}

\begin{example}\label{ex:threefold-purity}
If $X$ is a regular integral scheme with function field $F$, which
admits a proper surjective morphism $X \to \Spec R$, and $\alpha \in
\Br(X)$, then by the valuative criterion, the image of $\alpha$ under
the map $\Br(X) \to \Br(F)$ is unramified over $R$.  Conversely, by
purity for regular local rings (\cite{gabber:brauer}*{Theorem~2'} for
schemes of dimension at most $3$ and \cite{cesnavicius} in general),
any $\alpha \in \Br(F)$ that is unramified over $R$ is in the image of
the map $\Hoh^2(X,\Gm) \to \Br(F)$.
\end{example}

The following gives a useful criterion for checking that Brauer
classes become unramified after a finite extension.

\begin{lemma}
\label{lem:unramified_criterion}
Let $R$ be a commutative ring and $X$ a regular integral scheme with a
proper surjective morphism $X \to \Spec R$.  Let $F$ be the function
field of $X$ and $L/F$ a finite extension.  Let $\alpha \in \Br(F)$.
If for every point $x \in X$ with $R \subset \Oscr_{X,x}$,
there exists a regular ring $S \subset L$ that is an integral
extension of $\Oscr_{X,x}$, such that the image of $\alpha$ in $L$ is
unramified over~$S$, then the image of $\alpha$ in $\Br(L)$ is
unramified over $R$.
\end{lemma}

\begin{proof}
Let $A$ be a discrete valuation ring with fraction field $L$
containing $R$.  Then the intersection $A \cap F$ is a discrete
valuation ring with fraction field $F$ containing $R$.  By the
valuative criterion for properness, there exists an $R$-morphism
$\Spec A \cap F \to X$.  Then the image $x \in X$ of the closed point
is regular with $R \subset \Oscr_{X,x}$.  By our hypothesis, there
exists a regular ring $S \subset L$ that is an integral extension of
the regular local ring $\Oscr_{X,x}$ on which the image of $\alpha$ in
$\Br(L)$ is unramified over $S$.  Since $A$ is integral over $A\cap F$
and $\Oscr_{X,x} \subset A \cap F$, it follows that the integral
closure of $\Oscr_{X,x}$ in $L$ is contained in $A$. Hence $S$, being
integral over $\Oscr_{X,x}$, is contained in $A$.  Since $\alpha_L$ is
unramified over $S$, it is unramified at $A$ by definition.
\end{proof}

\subsection{Local description of ramification}
\label{sec:ramification}

Recall that a regular system of parameters in a regular local ring is
a minimal generating set of the maximal ideal.  A subsequence of a
regular system of parameters is called a \df{partial} regular system
of parameters.  Not every regular sequence is a partial regular system
of parameters.
We fix a positive integer $\ell$.

\begin{definition}
Let $R$ be a regular local ring with fraction field $F$ and assume
that $\ell$ is invertible in $R$. We say that $\alpha \in\Hoh^2(F,
\m_\ell^{\otimes 2})$ is \df{nicely ramified} if $\alpha$ is ramified
only along a partial regular system of parameters $x_1, \ldots, x_h$
of~$R$ and we can write
  \[
      \alpha = \alpha_0 + \sum_{i=1}^h (u_i, x_i) + \sum_{1\leq i<j\leq h} m_{i,j}(x_i, x_j)
  \]
  for an unramified class $\alpha_0$ and some $u_i\in R\mult$ and
  $m_{i,j} \in \ZZ$.

  More generally, if $X$ is a regular noetherian integral scheme with function
  field $F$ with $\ell$ invertible on $X$, and $\alpha \in\Hoh^2(F,
  \m_\ell^{\otimes 2})$, then we say that $\alpha$ is nicely ramified
  on $X$ if it is nicely ramified at every local ring of $X$.
\end{definition}

We will need the following result, proved in the two-dimensional case
in \cite{saltman-padic} and in the equicharacteristic case in
\cite{pirutka}*{Section~3,~Lemma~2}.

\begin{lemma}
  \label{nicely_ramified}
  Let $X$ be a regular noetherian integral scheme with function field $F$ and let
  ${\alpha \in\Hoh^2(F, \m_\ell^{\otimes 2})}$ where $\ell$ is
  invertible on $X$. If $\alpha$ is ramified only along a
  strict normal crossings divisor, then $\alpha$ is nicely ramified on
  $X$.
\end{lemma}

\begin{proof}
  Let $R$ be a local ring of $X$. By hypothesis, $\alpha$ is ramified
  only along a partial regular system of parameters $x_1, \ldots, x_h$. We
  proceed by induction on $h$. For $h = 1$, let $F_1$ be the fraction
  field of $R/(x_1)$ and let $\beta = (b) \in F_1\mult/F_1^{\times \ell} =
  \Hoh^1(F_1, \m_\ell)$ be the residue of $\alpha$, where $(b)$ denotes the class
  (or symbol) of $b$ in $F_1^\times/F_1^{\times\ell}$.  Then it follows that
  $\beta$ is unramified by considering the Gersten complex
  \[\Hoh^2(F, \m_\ell^{\otimes 2}) \to \bigoplus_{p \in \Spec(R)^{(1)}}
  \Hoh^1(k(p), \m_\ell) \to \bigoplus_{q \in
  \Spec(R)^{(2)}}\Hoh^0(k(q), \ZZ/\ell\ZZ).\]
  For the construction of the complex in this generality,
  see~\cite{kato}*{Section~1}.  Since $R/(x_1)$ is a regular local
  ring, it is a UFD, and we may write $b = \ol{u}_1\prod\pi_i^{e_i}$
  for irreducible elements $\pi_i$ and a unit $\ol{u}_1$ of
  $R/(x_1)$. Since the residue of $b$ at the prime $(\pi_i)$ is the
  class of $e_i$ modulo $\ell$, it follows that each $e_i$ is a
  multiple of $\ell$, and thus $\beta = (\ol{u}_1) \in
  F_1\mult/F_1^{\times \ell}$.

  Lifting $\ol{u}_1$ to a unit $u_1$ of $R$, it follows that the class
  $\alpha - (u_1, x_1)$ is unramified on $R$. In particular, we may
  write $\alpha = \alpha_0 + (u_1, x_1)$ for $\alpha_0$ unramified as
  claimed. For $h > 1$, let $\beta \in F_1\mult/F_1^{\times \ell}$ be
  the residue of $\alpha$ at the prime $(x_1)$, as before. Considering
  the Gersten complex, the residue of $\beta$ must be canceled by the
  residues of $\alpha$ along primes in $R/(x_1)$. In particular, it
  follows that $\beta$ can only be ramified along the primes $\ol{x}_2,
  \ldots, \ol{x}_h$ in $R/(x_1)$. Since $R/(x_1)$ is a regular local
  ring, it is a UFD, and we may represent $\beta$ by an element $\ol{b} =
  \overline{u}_1 \prod_{i=2}^h \overline x_i^{m_{i, 1}}$ with $\ol{u}_1$ a
  unit in $R/(x_1)$. In particular, we can lift $\ol{b}$ to $b = u_1 \prod_{i =
    2}^h x_i^{m_{i, 1}}$, where $u_1 \in R\mult$. It follows that
  \[\alpha - (b, x_1) = \alpha - (u_1, x_1)
    - \sum_{i = 2}^h m_{i, 1} (x_i, x_1)\]
  is unramified along $(x_1)$ and only ramified along the primes $(x_i)$ for
  $i = 2, \ldots, h$. By induction, we may write
  \[\alpha - (b,x_1) = \alpha_0  + \sum_{j = 2}^h (u_j, x_j) + \sum_{j, k \neq 1} m_{j, k} (x_j, x_k), \]
  yielding $\alpha = \alpha_0 + \sum (u_i, x_i) + \sum m_{i,j} (x_j, x_k)$ as
  desired.
\end{proof}

\subsection{Putting ramification in nice position}

We will need the following generalization of \cite{pirutka}*{Lemma~3},
which from the toroidal geometry perspective, is related to the
process of barycentric subdivision.  The standard reference for
toroidal geometry is \cite{KKMS}, which is written over an
algebraically closed base field.  However, all the constructions work
over an arbitrary base scheme as outlined in
\cite{FaltingsChai}*{IV~Remark~2.6}.

\begin{definition}
    Let $D \subset X$ be a strict normal crossings divisor in a
    regular noetherian scheme. We define a presentation of $D$ to be a
    finite collection $\{D_i\}_{i\in I}$ of regular, but not necessarily connected, divisors
    such that $D=\cup_{i\in I}D_i$ and $D_i=D_j$ implies $i=j$. We call $|I|$
    the length of the presentation.
\end{definition}

For example, we may choose our presentation to simply consist of the
irreducible components of $D$. On the other hand, the next lemma shows that
after possibly blowing up, we may find a presentation whose length is
bounded by the dimension of $X$.

\begin{lemma}
  \label{colouring lemma}
  Let $X$ be a regular noetherian scheme of dimension $d$ and
  suppose that $D \subset X$ is a strict normal crossings
  divisor. There exists a sequence of blowups along regular subschemes
  {$f: X' \to X$} such that $f^{-1}(D)$ admits a presentation of length at
  most $d$.
\end{lemma}

Before giving the proof, we recall some combinatorial notions related
to inverse images of strict normal crossings divisors under certain blowups.

\begin{definition}
An \df{(abstract) simplicial complex} $\Sigma$ is a collection of
nonempty finite sets, called \df{simplices}, closed under inclusion.
The union of all simplices is the \df{vertex set} of $\Sigma$. 
    \begin{enumerate}
    \item The elements $\sigma\in\Sigma$ of cardinality $i+1$ are
    called \df{$i$-simplices}.  We write $\Sigma_i$ for the subset of
    all $i$-simplices in $\Sigma$. By abuse of notation, we also use
    the term vertex for a $0$-simplex and the symbol $\Sigma_0$ for
    the vertex set.
    \item The \df{dimension} of $\Sigma$ is defined to be the
    maximal $i \geq 0$ such that $\Sigma_i \neq \varnothing$, assuming
    that this maximum exists.
    \item Given a simplicial complex and a non-empty simplex
    $\sigma \subset \Sigma$, we define the \df{star subdivision}
    $\Sigma \star \sigma$ with respect to $\sigma$ to be the
    simplicial complex whose vertex set is the vertex set of $\Sigma$
    together with a new vertex $e_\sigma$, and whose simplices are
            \begin{gather*}
            \{\tau \in \Sigma \;:\; \sigma \not\subseteq \tau \}
            {\;\textstyle\bigcup\;} \{ (\tau
            \smallsetminus J) \cup \{e_\sigma\} \;:\; \emptyset\neq J \subseteq \sigma
            \subseteq \tau \in \Sigma\} \\
            = \{\tau \in \Sigma \;:\; \sigma \not\subseteq \tau \}
            {\;\textstyle\bigcup\;} \{ (\tau' \cup \{e_\sigma\}
            \;:\; \emptyset \neq J \subseteq \sigma \subseteq \tau' \cup J
            \in \Sigma \text{, some } J \subseteq \Sigma_0 \setminus \tau'\}
            .
            \end{gather*}
        \item We formally define $\Sigma \star \varnothing = \Sigma$.
    \end{enumerate}
\end{definition}

\begin{remark}
Let $\Sigma$ be a simplicial complex.
    \begin{enumerate}
    \item[(i)] If $\sigma$ is a $0$-simplex, then $\Sigma \star \sigma$ is
    isomorphic to $\Sigma$ by sending $e_\sigma$ to $\sigma$.
    \item[(ii)] If $\sigma$ and $\tau$ are simplices such that neither
    $\sigma\subseteq\tau$ nor $\tau\subseteq\sigma$, then $(\Sigma
    \star \sigma) \star \tau = (\Sigma \star \tau) \star \sigma$.  In
    particular, for any subset $\Sigma' \subseteq \Sigma$ consisting
    of simplices none of which contain any other, we can define the
    iterated star subdivision $\Sigma \star \Sigma'$ with respect to
    all simplices in $\Sigma'$.
    \end{enumerate}
\end{remark}

\begin{definition} Let $\Sigma$ be a simplicial complex.
    \begin{enumerate}
        \item[(a)] The \df{barycentric subdivision} $\mathrm{Sd}(\Sigma)$ of $\Sigma$ is the iterated star
            subdivision $$(((\Sigma \star \Sigma_d) \star \Sigma_{d-1}) \star
            \dotsm ) \star \Sigma_1,$$ see \cite{KKMS}*{III,~\S2A}.
        \item[(b)] The \df{order complex} $\mathrm{Fl}(\Sigma)$ of $\Sigma$ is the simplicial complex with a vertex for each simplex of
            $\Sigma$ and whose $i$-simplices are all length $i+1$ flags of
            inclusions of simplices of $\Sigma$.
    \end{enumerate}
\end{definition}

We will need the following combinatorial lemma.

\begin{lemma}\label{lem:baryorder}
    There is a natural isomorphism of simplicial complexes
    $\mathrm{Sd}(\Sigma)\iso\mathrm{Fl}(\Sigma)$.
\end{lemma}

\begin{proof}
    See \cite{kozlov}*{\S2.1.5}.
\end{proof}

Now, we apply these definitions to a simplicial complex associated to
a presentation of a strict normal crossings divisor.

\begin{definition}
Let $D \subset X$ be a strict normal crossings divisor in a regular
noetherian scheme and let $\{D_i\}_{i \in I}$ be a presentation of
$D$. We define a simplicial complex $\Sigma(D) = \Sigma(D, \{D_i\}_{i
\in I})$ with vertex set $I$ such that a subset $J \subset I$ is in
$\Sigma(D)$ whenever $\cap_{j \in J} D_j \neq \emptyset$. We call
$\Sigma(D)$ the \df{naive dual complex} associated to the
presentation.
\end{definition}

\begin{remark}
The usual dual complex, which includes distinct
$i$-simplices for each irreducible component of each intersection of
$i+1$ components of $D$, has better homotopical properties (e.g., see
\cite{payne}); the naive version will suffice for our purposes.
\end{remark}

Following standard conventions, for $J \subset I$, we will
let $D_J$ denote the intersection $\cap_{j \in J} D_j$. By
hypothesis, $D_J$ is nonempty whenever $J \in \Sigma(D)$ and since $D$ is snc
is a regular subscheme of $X$ of codimension $|J|$.

\begin{remark}
    The dimension of $\Sigma(D)$ is bounded above by the dimension of $X$.
\end{remark}


An important fact relating the geometry of the pair $(X,D)$ to the
combinatorics of the dual complex $\Sigma(D)$ is the following.

\begin{lemma}\label{lem:blowup}
    Let $\{D_i\}_{i\in I}$ be a presentation of a snc divisor $D$ in
    a regular noetherian scheme $X$. Let $\sigma$ be an $i$-simplex of
    $\Sigma(D)$ and let $f\colon\mathrm{Bl}_{D_\sigma}X\rightarrow X$ be the
    blowup of $X$ along~$D_\sigma$. We let
    \begin{itemize}
        \item   $\til D=f^{-1}(D)$ denote the inverse image of $D$,
        \item   $\til D_i=f^{\mathrm{st}}(D_i)$ denote the strict transform of $D_i$ for $i\in I$,
        \item    and $E$ denote the exceptional divisor of $f$.
    \end{itemize}
    The naive dual complex $\Sigma(\til D)$ with respect to the presentation $\{E\}\sqcup\{\til D_i\}_{i \in I}$ of
    $\til D$ is naturally isomorphic to $\Sigma(D) \star \sigma$, where the
    new vertex $e_\sigma$ corresponds to $E$.
\end{lemma}

Note that an analogous statement is made in
\cite{CoxLittleSchenck}*{Proposition~3.3.15} in the case of the usual
dual complex.

\begin{proof}
We use the natural bijection of vertex sets
$\{\Sigma(D)\star\sigma\}_0\iso\Sigma(\til D)_0$ which is equality on
$I=\Sigma(D)_0$ and sends $e_\sigma$ to $E$.  To prove the lemma, we
need to show that the incidence of the divisors $E$ and $\til D_i$
satisfy the same incidence relations as the $0$-simplices of the
complex $\Sigma(D) \star \sigma$, namely that
    \begin{enumerate}
        \item for $J \subseteq I$, $\til D_J \neq \emptyset$ if and only if $D_J
            \neq \emptyset$ and $\sigma \not\subseteq J$;
        \item for $J = J' \sqcup \{e_\sigma\}$ where $J' \subset I$, $\til D_J
            \neq \emptyset$ if and only if $D_{J' \sqcup J''} \neq \emptyset$ for
            some $J'' \subseteq \sigma \subseteq J' \cup J''$ with $\emptyset \neq
            J'' \subseteq I \setminus J'$.
    \end{enumerate}
    Geometrically, the first condition says that $\cap_{i\in J}\til D_i$ will
    be nonempty if and only if $\cap_{i\in J}D_i$ is nonempty and is not
    contained in $D_\sigma$. The only nontrivial part
    of the statement then is the assertion that if the intersection $D_J$ is
    nonempty and is contained in $D_\sigma$, then $\til D_J$ is empty. To see
    this, note that $\til D_J \subset \til D_\sigma$, so it suffices to assume
    that $J = \sigma$. Looking \'etale locally near $D_\sigma$, we can replace
    $X$ with $\Spec R$ for a regular local ring $R$, and where $D_i, i \in
    \sigma$ is cut out by $x_1, \ldots x_m$, which form part of a regular
    system of parameters. The blowup is then given as the relative proj of the
    graded ring $R[t_1, \ldots, t_m]/(x_i t_j - x_j t_i)$ for pairs
    of indices 
    $i \neq j$, with
    the strict transform $\til D_i$ cut out by the ideal $(t_i, x_i)$. It
    follows that the stratum $\til D_\sigma$ is defined by an ideal which
    contains the irrelevant ideal, and is therefore empty.

    The content of the second part is the statement that a stratum $\til
    D_{J'}$, where $J' \not\subseteq \sigma$ (this condition is ensured by $J''
    \subset \sigma \setminus J'$) should have a nonempty intersection with the
    exceptional divisor $E$ exactly when the stratum $D_{J'}$ has a nontrivial
    intersection with a stratum $D_{J''}$ with mutual intersection $D_{J' \sqcup
    J''}$ contained in the blowup locus $D_\sigma$.

    Suppose that $J'$ and $J''$ are chosen as above. As $\sigma \not\subseteq J'$
    (since the elements of $J''$ are contained in $\sigma$ but not in $J'$),
    $D_{J'} \not\subseteq D_\sigma$. On the other hand $D_\sigma \cap D_{J'}
    \supseteq D_{J''} \cap D_{J'} = D_{J'' \sqcup J'} \neq \emptyset$, and hence
    $D_J'$ nontrivially intersects the locus $D_\sigma$ which is being blown
    up. But therefore $\til D_{J'}$, which may be identified with the strict
    transform of $D_{J'}$ nontrivially intersects $E$, as desired.

    In the other direction, if for some $J' \not\subseteq \sigma$, all such
    intersections were trivial, then it would follow that $D_{J'\sqcup \sigma}$
    is also trivial, showing that $D_{J'}$ is disjoint from
    $D_\sigma$. But in this case, it is easy to see that $D_{J'}$ will be disjoint from $E$ as claimed.
\end{proof}

Now, we can use Lemma~\ref{lem:blowup} to prove Lemma~\ref{colouring lemma}.

\begin{proof}[Proof of Lemma~\ref{colouring lemma}]
    Consider the proper birational morphism $f : X' \to X$ obtained by
    sequentially blowing up the strata of $D$, first blowing up the
    $0$-dimensional strata of $D$, then the strict transforms of the
    $1$-dimensional strata of $D$, then the strict transforms of the
    $2$-dimensional strata of $D$, etc.  The inverse image $f^{-1}(D)$ is a strict
    normal crossings divisor in $X'$ and, by Lemma~\ref{lem:blowup}, its
    naive dual complex is the barycentric subdivision of $\Sigma(D)$.
    Given its interpretation in Lemma~\ref{lem:baryorder} as the order complex of $\Sigma(D)$, the
    vertices of the barycentric subdivision can be colored with at most $d$
    colors ``by dimension'', with the subset of vertices corresponding to
    elements in $\Sigma(D)_i$ having color $i$.  This coloring has the
    property that no two distinct vertices of the same color are both contained in
    a simplex; equivalently, the correspondingly colored components of
    $f^{-1}(D)$ are disjoint in $X'$.  Hence the union of all irreducible
    components of the same color is a regular divisor.
    Therefore, we can express $f^{-1}(D)$ as the union of at most $d$
    regular (but not necessarily connected) divisors, as required.  If
    need be, we can further blow up smooth points to get a union of
    exactly $d$ regular divisors.
\end{proof}

\subsection{Construction of rational functions for splitting ramification}

\begin{notation}
  Let
  $\{V_i \}_{i \in I}$
  be a family of cycles on $X$. Given a subset $J \subset I$, let
  $V_J$ denote the naive intersection cycle, defined as follows. Given
  two integral subschemes $A$ and $B$ of $X$, the naive intersection cycle
  is $(A\cap B)_{\text{\rm red}}$, written as a sum of its irreducible
  components. Given two cycles $\sum a_iW_i$ and $\sum b_jW_j$, the
  naive intersection is the union of the naive intersections $W_i\cap
  W_j$. (This is purely a way of measuring dimensions of support as a
  notational convenience,
  nothing else. We thus ignore coefficients and intersection multiplicities.)
\end{notation}

\begin{definition}
  A collection of irreducible subschemes $\{W_i\}_{i \in I}$ of $X$
  \df{intersects properly} if for every subset $J \subset I$ we have
  $\codim W_J \geq \sum_{i \in J} \codim W_i$ (using the convention
  that $\codim \emptyset = \infty$). A collection of cycles
  $\{V_i\}_{i \in I}$ on $X$ intersect properly if any collection
  $\{W_i\}_{i\in I}$, where $W_i$ is an irreducible component of the
  support of $V_i$ for all $i \in I$, intersects properly.
\end{definition}

\begin{lemma}
    Let $X$ be a scheme.
  Suppose that $\{W_i\}_{i \in I}$ is a collection of cycles of $X$ that
  intersect properly. If $W \subset X$ is an irreducible
  subscheme such that, for every subset $J \subset I$, the scheme $W$
  intersects each irreducible component of $W_J$ properly, then
  $\{W_i\} \cup \{W\}$ intersects properly.
\end{lemma}
\begin{proof}
  We omit the proof.
\end{proof}

Now, we prove a lemma which is a direct generalization of the lemma in the
correction to~\cite{saltman-padic}, and which can be viewed as a
version of Kawamata's trick, see \cite{abramovich_karu}*{\S5.3}.

\begin{lemma}
  \label{proper_lemma}
  Let $X$ be a regular scheme admitting an ample
  invertible sheaf and let $D_1, \ldots, D_d$ be a collection of
  regular divisors on $X$ whose union is a strict normal crossings
  divisor. Let
  \[ \til D_i = \sum_{j = 1}^d m_{i,j} D_j, \ \ i = 1, \dotsc, n\]
  be a collection of integer linear combinations of the $D_j$. In this case,
  for $i=1, \dotsc, n$, there exist rational functions $f_i$ and
  divisors $E_i$ on $X$ such that
  \begin{enumerate}
  \item $\operatorname{div}(f_i) = \til D_i + E_i$, and

  \item the collection $\{E_i\}_{i=1,\dotsc,n} \cup \{D_j\}_{j=1,\dotsc,d}$ intersects properly.
  \end{enumerate}
\end{lemma}
\begin{proof}
  We construct the $f_i$ inductively.

\smallskip

  \noindent
  \textbf{Base case $i = 1$}. Let $P$ be a (scheme-theoretic) disjoint
  union of closed points, with one closed point in each irreducible
  component of $D_I$ for each subset $I \subset \{1,\dotsc,d\}$.  Let
  $R$ be the semilocal ring of the points of $P$ (which exists since
  we can put $P$ in a single quasi-affine open of $X$, using the ample
  invertible sheaf \cite{MR0217084}*{Th\'eor\`eme 4.5.2}
  and the graded prime-avoidance lemma \cite{eisenbud}*{Section~3.2}).
  By
  \cite{bourbaki:commutative_algebra1-4}*{Chapitre~II,~Section~5.4,~Proposition~5},
  since finitely generated projective modules over $R$ are free, it
  follows that each $D_i$ is principal on $R$. In particular, we may
  write $D_i$ as the zero locus of a function $x_i \in R$. We then
  have, upon setting $f_1 = \prod x_j^{m_{1, j}}$, that $(f_1) = \til
  D_1 + E_1$ and the support of $E_1$ contains none of the strata
  $D_I$. In particular, since $E_1$ is a divisor, the codimension of
  $E_1 \cap D_I$ in $D_I$ is at least $1$ as desired.

\smallskip

  \noindent
  \textbf{Induction step}. Suppose we have previously defined $f_1,
  \ldots, f_{r-1}$ so that $\operatorname{div}(f_i) = \til D_i + E_i$
  with $\{E_i\}_{i = 1, \dotsc, r - 1}\cup\{D_j\}_{j=1,\dotsc,d}$
  intersecting properly.  For every pair of subsets $I \subset \{1,
  \ldots, d\}$ and $J \subset \{1, \ldots, r -1\}$, consider the
  intersection $D_I \cap E_J$ and let $P$ be a scheme theoretic union
  consisting of at least one closed point from each irreducible
  component of each of these nonempty intersections for every $I, J$
  as above. Let $R$ be the semilocal ring at $P$. As above, we may
  write $D_i$ as the zero locus of some function $x_i \in R$ on $\Spec
  R \subset X$, and considering the $x_i$ as rational functions on
  $X$, we set $f_r = \prod x_j^{m_{r, j}}$. It follows that $(f_r) =
  \til D_r + E_r$ where the support of $E_r$ contains no irreducible
  component of the strata $D_I \cap E_J$, with $J \subset \{1, \ldots,
  r - 1\}$. Therefore $\{E_i\}_{i = 1, \ldots,
  r}\cup\{D_j\}_{j=1,\dotsc,d}$ intersects properly as desired.
\end{proof}

\begin{notation}
  Let $T \in \Mat_{n,d}(\ZZ)$ be an $n \times d$ matrix. For subsets $I
  \subset \{1, \ldots n\}$ and $J \subset \{1, \ldots, d\}$, let
  $T_{I,J}$ denote the $|I|\times|J|$-submatrix of $T$ with rows
  corresponding to the elements of $I$ and columns corresponding to the
  elements of $J$.
\end{notation}

\begin{definition}
  Let $\ell$ be a prime and $n, d$ be positive integers with $n \geq
  d$. An $n \times d$ matrix $T \in \Mat_{n,d}(\ZZ)$ is
  \df{$\ell$-Pirutka} if for all nonempty subsets $I \subset \{1,
  \ldots, n\}$ and $J \subset \{1, \ldots, d\}$, with $|I| - |J| = n -
  d$, the submatrix $T_{I, J}$ has (maximal) rank $|J|$ modulo $\ell$.
\end{definition}

\begin{lemma}
  \label{dantarded}
  Let $R$ be a regular local ring with
  fraction field $F$ and $\alpha \in\Hoh^t(F, \m_\ell^{\otimes t})$ where
  $\ell$ is invertible in $R$. Let
  $x_1, \ldots, x_n \in R$ be a regular system of parameters and suppose
  that $\alpha = (u_1, \ldots, u_{t - h}, x_1, \ldots, x_h)$ with
  $u_i$ units in $R$. If $L = F(\sqrt[\ell]{x_1}, \ldots,
  \sqrt[\ell]{x_h})$ and $S$ is the integral closure of $R$ in $L$,
  then
  \begin{enumerate}
  \item $S$ is a regular local ring with maximal ideal generated by
    $\sqrt[\ell]{x_1}, \ldots, \sqrt[\ell]{x_h}, x_{h+1}, \ldots, x_n$,
  \item the class $\alpha_L$ has trivial residue at each codimension one prime
    of $S$.
  \end{enumerate}
\end{lemma}
\begin{proof}
  We omit the proof that $S =
  R[z_1,\dotsc,z_h]/(z_1^\ell-x_1,\dotsc,z_h^\ell-x_h)$ and is regular
  with maximal ideal $\mathfrak m = (z_1, \dotsc, z_h,x_{h+1},\dotsc,x_n)$.

  \newcommand{\res}{\operatorname{res}}

  Now, for a prime $\mathfrak P \subset R$ of height one and a prime
  $\mathfrak Q \subset S$ lying over it with ramification index $e$, we
  have a commutative diagram
  \[\xymatrix{
      \Hoh^t(F, \m_\ell^{\otimes t}) \ar[d] \ar[r] &
      \Hoh^{t-1}(\mathrm{Frac}({R/\mathfrak P}),
      \m_\ell^{\otimes t-1}) \ar[d]_{e \res} \\
      \Hoh^t(L, \m_\ell^{\otimes t}) \ar[r] & \Hoh^{t-1}(\mathrm{Frac}({S/\mathfrak
      Q}),
      \m_\ell^{\otimes t-1})
    }\]
  of residue maps, which shows that $\alpha_L$ only ramifies at primes lying over the
  ramification locus of $\alpha$. In particular, $\alpha_L$ can only
  ramify over the primes $(z_i)$ for $i = 1, \ldots, h$, which each have
  ramification index $\ell$ over $(x_i)$. Since all residues of $\alpha$
  are $\ell$-torsion, it follows from the diagram above that $\alpha_L$
  is unramified.
\end{proof}

\begin{lemma}
  \label{main lemma}
  Let $X$ be a regular
  scheme of dimension $d$ admitting an ample invertible sheaf and $D_1,\dotsc,D_d$ be a
  collection of regular divisors of $X$ whose union is snc. Let $\ell$ be a
  prime invertible on $X$ and let $\alpha
  \in\Hoh^2(F, \m_\ell^{\tensor 2})$ be a class ramified only along the
  union of the $D_i$.  Let $T = (m_{ij})$ be an $\ell$-Pirutka $n \times
  d$ matrix.  For each $i=1,\dotsc,n$, let $\til D_i = \sum_{j=1}^d
  m_{ij} D_j$ and let $f_i$ and $E_i$ be as in Lemma~\ref{proper_lemma}.
  If $L = F(\sqrt[\ell]{f_1}, \ldots, \sqrt[\ell]{f_n})$, then
  $\alpha_L$ is unramified.
\end{lemma}
\begin{proof}
  By
  Lemma~\ref{nicely_ramified}, for any point $z\in X$,
  we have that
  \begin{equation}\label{asher}
    \alpha = \alpha_0 + \sum (u_i, x_i) + \sum_{i,j} m_{i,j}(x_i, x_j)
  \end{equation}
  for $u_i\in\Oscr_{X,z}\mult$ and $x_i$ local equations for $D_i$ in $\Oscr_{X,z}$.

  To show that $\alpha$ becomes unramified in $L$, by
  Lemma~\ref{lem:unramified_criterion}, it suffices to show that for
  every point $z\in X$ and each term in \eqref{asher}, there is a
  subfield of $L$ where that term becomes unramified over some regular
  subring contained in $L$ and integral over $\Oscr_{X,z}$.
  For example, by Lemma~\ref{dantarded}, any term of the form
  $(u_i,x_i)$ will become unramified over a regular subring of an
  extension $F(\sqrt[\ell]{g_i})$ where $g_i$ is a local equation for
  $D_i$ at $z$; any term of the form $(x_i, x_j)$ will become
  unramified over a regular subring of an extension $F(\sqrt[\ell]{g_i},
  \sqrt[\ell]{g_j})$.

  We thus seek the following: for each point $z \in X$ and each
  $j=1,\dotsc,d$, an element $g_j\in F$ such that
  \begin{enumerate}
  \item $g_j$ is a local equation for $rD_j$ at $z$, where $r\equiv 1\bmod\ell$;
  \item $F(\sqrt[\ell]{g_j})\subset L$.
  \end{enumerate}
  Choose $J$ maximal with respect to inclusion so that $z \in D_J$. If
  $J=\emptyset$ (so that $\alpha$ is unramified over $\Oscr_{X,z}$),
  then $g_j=1$ works for all $j$. Otherwise, choose any $j_0 \in J$;
  we will find $g_{j_0} \in F$ satisfying conditions (1) and (2)
  above.

  We claim that we can find $I \subset \{1, \ldots n\}$ with $|I| -
  |J| = n - d$ and $z \not\in \cup E_i$. To see this, set
  \[ I' = \{i \in \{1, \ldots, n\} \,|\, z \in E_i \}. \]
  Since $z \in E_{I'} \cap D_J$, it follows by the properness of the
  intersection that $|I'| + |J| \leq d = \dim(X)$. In particular,
  there are at most $d - |J|$ indices $i$ such that $z \in E_i$. This
  means we can find a set of $n - (d - |J|)$ indices $i$ such that $z
  \not\in E_i$. Let $I$ be such a set.  Since $T$ is an $\ell$-Pirutka
  matrix, the submatrix $T_{I,J}$ has full rank $|J|$ modulo $\ell$,
  and hence we can find $a_i \in \ZZ$ for $i \in I$ such that $\sum_{i
  \in I} a_i m_{i,j} \equiv \delta_{j,j_0} \bmod \ell$ for each $j \in
  J$. Translating in terms of $\til D_i$ and $D_i$, this says that
  there exists $r \equiv 1 \bmod \ell$ such that
  \[
    \sum_{i \in I} a_i \til D_i = r D_{j_0} + \sum_{j \not\in J} b_j D_j
  \]
  and therefore
  \[
    \operatorname{div}\left(\prod_{i \in I} f_i^{a_i}\right) = r D_{j_0} + \sum_{j \not\in J} b_j D_j +
    \sum_{i \in I} a_i E_i.
  \]
  Let $g_{j_0} = \prod_{i \in I} f_i^{a_i}$. Since $z \not\in E_i$ for
  all $i \in I$ and $z \not\in D_j$ for all $j \not\in J$, we find that $g_{j_0}$
  is a local equation for $rD_{j_0}$ in $\Oscr_{X,z}$. It is clear that
  $\sqrt[\ell]{g_{j_0}}\in L$.
\end{proof}

\begin{theorem}\label{clever-splitsville}
Let $X$ be a regular scheme of dimension $d$
admitting an ample invertible sheaf. Let $\ell$ be a prime invertible
on $X$ and $\alpha \in\Hoh^2(F, \m_\ell^{\otimes 2})$ be ramified
along a strict normal crossings divisor.  If there is an
$\ell$-Pirutka $n \times d$ matrix, then we can find rational
functions $f_1, \ldots, f_n \in F$ so that $\alpha$ becomes unramified
in $L = F(\sqrt[\ell]{f_1}, \ldots, \sqrt[\ell]{f_n})$.
\end{theorem}
\begin{proof}
  By Lemma~\ref{colouring lemma} we can perform a sequence of blowups to $X$
  so as to make the ramification divisor of $\alpha$ contained in
  an snc divisor that we can write as a union $D_1 \cup \cdots
  \cup D_d$ of regular divisors. Now the result is an immediate application of
  Lemma \ref{main lemma}.
\end{proof}

\subsection{Some Pirutka matrices}
\label{sec:some-pirutka-matr}
\begin{example}
  \label{ex:pirutka_d_by_dsquared}
  Pirutka's proof in~\cite{pirutka} uses the following $\ell$-Pirutka matrix.
  Consider $n = d^2$, and let $T$ be the $d^2\times d$ matrix given by $d$ (vertical) copies
  of the $d \times d$ identity matrix. The condition is now: for every subset of columns $J$, and
  subset of rows $I$ of order $d^2 - d + |J|$, we have full rank. But notice
  that since $|J| \geq 1$, we are always removing fewer than $d$ rows. Since
  each row of the identity matrix occurs $d$ times, each row of the identity
  matrix must still occur in the $I,J$-minor, showing that $T_{I,J}$ has full rank.
\end{example}

\begin{example}\label{3x3_clever}
  The $3\times 3$ matrix
  $$
  \begin{pmatrix}
    1 & 3 & 3 \\
    1 & 2 & 1 \\
    1 & 1 & 2
  \end{pmatrix}
  $$
  considered in \cite{pirutka}*{Remark~5} is $\ell$-Pirutka for all
  primes $\ell>3$.
\end{example}

\begin{example}\label{4x3_cheeseburger}
  The $4\times 3$ matrix
  $$
  \begin{pmatrix}
    1 & 1 & 1\\
    1 & 1 & 0\\
    0 & 1 & 1\\
    1 & 2 & 1
  \end{pmatrix}
  $$
  found in \cite{pirutka}*{Remark~4}
  is $\ell$-Pirutka for all primes $\ell$.
\end{example}

\begin{remark}
  A computer search shows that
  \begin{enumerate}
  \item there are no $2$-Pirutka $2\times 2$ or larger square matrices, and
  \item there are no $3$-Pirutka $3\times 3$ or larger square matrices.
  \end{enumerate}

  It is easy to make a $2\times 2$ matrix that is $\ell$-Pirutka for
  all primes $\ell > 2$. This allows one to split the ramification of
  classes of odd prime period on surfaces using roots of two rational functions,
  which reproduces results of
  Saltman~\cite{saltman:cyclic_algebra_p-adic_curves} except for
  classes of period $2$.
\end{remark}

\begin{question}
  For which $\ell,n,d$ do there exist $\ell$-Pirutka $n \times d$ matrices?
\end{question}

It is clear that if $n \gg d$, then there exist $\ell$-Pirutka $n\times d$
matrices. We also have the following bound for square matrices.

\begin{proposition}
If $\ell > \binom{2n-1}{n}$ is prime, then there exist $\ell$-Pirutka
$n \times n$ matrices.
\end{proposition}

\begin{proof}
First note that an $n\times n$ matrix $T$ is $\ell$-Pirutka if and
only if all maximal minors of the $n \times 2n$ matrix $A = (I_n|T)$
do not vanish modulo $\ell$.  We will consider building the matrix $A
= (e_1,\ldots,e_n,t_1,\ldots,t_n)$ by inserting the columns $t_i$ one
at a time.  For inserting the first column, we simply require that all
entries in $t_1$ do not vanish.  Once the first column has been fixed,
we require that $t_2$ avoids the $\binom{n+1}{n-1}$ hyperplanes
defined by the maximal minors containing $t_2$, which is certainly
possible if $\ell > \binom{n+1}{n-1}$.  Similarly, once the first $k$
columns have been fixed, we then require that $t_{k+1}$ avoids the
$\binom{n+k}{n-1}$ hyperplanes defined by the maximal minors
containing $t_{k+1}$, which is certainly possible if $\ell >
\binom{n+k}{n-1}$. When $k = n-1$, then inserting the final column
$t_n$ is certainly possible if the stated bound is satisfied.
\end{proof}

Of course this bound is far from sharp.  The hyperplanes in the above
proof are not in general position.

\section{Alterations}
\label{sec:alterations}
\numberwithin{theorem}{section}

In this section, we use Gabber's theory of prime-to-$\ell$ alterations
over a discrete valuation ring; see
\cite{illusie:Gabber_refined_uniformization},
\cite{illusie_laszlo_orgogozo}.  For an example of the statement we
are interested in, see \cite{colliot:bourbaki09}*{Th\'eor\`eme~3.25} and its
proof.

\begin{definition}
  Let $\ell$ be a prime number and $X$ a scheme of finite type over an
  excellent ring.  An \df{$\ell'$-alteration} $X' \to X$ is a proper
  surjective generically finite map such that for every maximal point
  $\eta$ of $X$, there exists a maximal point $\eta'$ of $X'$ over
  $\eta$ such that the residue field extension
  $\kappa(\eta')/\kappa(\eta)$ has degree prime to $\ell$.
\end{definition}

\begin{lemma}
  Let $X$ be an integral scheme, $X' \to X$ an $\ell'$-alteration,
  $\eta'$ a maximal point of $X'$ dominating $X$, and $\alpha \in
  \Br(\kappa(X))$.  If $\ind(\alpha_{\kappa(\eta')})=\ell^N$ then
  $\ind(\alpha)=\ell^N$.
\end{lemma}
\begin{proof}
  Because $\kappa(\eta')/\kappa(X)$ has degree prime-to-$\ell$, the
  result follows by a standard restriction-corestriction argument.
\end{proof}

\begin{definition}
Let $R$ be a discrete valuation ring, $s\in\Spec R$ the closed point,
and $X$ an $R$-scheme.  If $X$ is equidimensional, flat and of finite
type over $R$, the generic fiber of $X$ over $\Spec R$ is smooth, and
the reduced special fiber $X_{0,\red}$ is a strict normal crossings
divisor on $X$, then $X$ is said to be \df{quasi-semistable} over $R$.
\end{definition}

The following two results are a distillation of the main results of Gabber's
theory of uniformization by $\ell'$-alterations; see
\cite{illusie:Gabber_refined_uniformization}*{Theorem~1.4},
\cite{illusie_laszlo_orgogozo}*{X.2}, and
\cite{colliot:bourbaki09}*{Th\'eor\`eme~3.25}.

\begin{lemma}\label{lem:localform}
  If $X$ is quasi-semistable over $R$, then $X\rightarrow\Spec R$ is \'etale
  locally of the form
  \begin{equation}\label{eq:localform}
    X=\Spec R[t_1,\ldots,t_n]/(t_1^{a_1}\cdots t_r^{a_r}-\pi),
  \end{equation}
  where $\pi$ is a uniformizing parameter of $R$.
\end{lemma}

\begin{theorem}[Gabber]\label{thm:gabber}
  If $R$ is an excellent henselian discrete valuation ring with residue field
  $k$ of characteristic $p\geq 0$ and fraction field $K$, and $X$ is a proper
  scheme over $R$, then for any prime $\ell \neq p$, there exists a commutative
  diagram of $\ell'$-alterations
  \begin{equation*}
    \xymatrix{
      X'\ar[d]\ar[r]  &   X\ar[d]\\
      \Spec R'\ar[r]    &   \Spec R
    }
  \end{equation*}
  with $R'$ an excellent henselian discrete valuation ring such that
  $X'$ is a regular scheme that is quasi-semistable and projective
  over $R'$.
\end{theorem}

\begin{proposition}\label{prop:diagram}
Let $R$ be an excellent henselian discrete valuation ring, $X$ be an
integral scheme proper over $R$ of relative dimension $d$. If
$\alpha \in \Br(\kappa(X))[\ell]$, then there exists a diagram of
morphisms
  \begin{equation}
    \label{eq:Gabrutka}
    \xymatrix{
    Y'\ar[r]^h \ar[d] & Y \ar[dr]\ar[rr]^g & & X' \ar[dl] \ar[r]^{f} & X \ar[d] \\
      \Spec R''\ar[rr]  & & \Spec R' \ar[rr] & & \Spec R
    }
  \end{equation}
  where
  \begin{enumerate}
  \item $R'/R$ and $R''/R'$ are finite extensions of excellent
    henselian discrete valuation rings such that $R''/R$ has degree
    prime to $\ell$,

  \item $f$ and $h$ are $\ell'$-alterations,

  \item $X'$ and $Y'$ are regular and integral,

  \item $X' \to \Spec R'$ and $Y' \to \Spec R''$ are projective and
    quasi-semistable,

  \item $Y$ is integral and the function field extension induced by
    $g$ has the form
    $$\kappa(Y) = \kappa(X')(\sqrt[\ell]{f_1},\dotsc,\sqrt[\ell]{f_N})$$
    for some $N$, and

\item $\alpha_{\kappa(Y')}$ lies in the subgroup
    $\Br(Y')[\ell]\subseteq\Br(\kappa(Y'))[\ell]$.
  \end{enumerate}
  Moreover, if there exists an $\ell$-Pirutka $n \times (d+1)$ matrix, then
  we may take $N=n$.
\end{proposition}
\begin{proof}
  By Theorem~\ref{thm:gabber}, there exists a commutative diagram of
  $\ell'$-alterations
  $$
  \xymatrix{
  X_1\ar[d]\ar[r]^{f_1}  &   X \ar[d]\\
  \Spec R_1\ar[r]    &   \Spec R
  }
  $$
  where $X_1$ is regular and integral, and projective over $R_1$.
  Consider the ramification divisor $D_1$ of $\alpha_{\kappa(X_1)}$ on
  $X_1$.  By an application of Gabber's embedded uniformization (see
  \cite{illusie:Gabber_refined_uniformization}*{Theorem~1.4}), there
  exists a further commutative diagram of $\ell'$-alterations
  $$
  \xymatrix{
  X'\ar[d]\ar[r]^{f_2}  &   X_1\ar[d]\\
  \Spec R'\ar[r]    &   \Spec R_1
  }
  $$
  where $X'$ is regular and integral, and projective and
  quasi-semistable over $R'$, and such that $f_2^{-1}(D_1)_\red\cup
  (X_s)_\red$ has normal crossings (but not necessarily strict normal
  crossings), where $s\in\Spec R'$ is the closed point. After blowing
  up $X'$, we may assume that $f_2^{-1}(D_1)_\red\cup (X'_s)_\red$ is
  a strict normal crossings divisor
  \cite{dejong-alterations}*{Paragraph~2.4} (see also
  \cite{brian}). On the other hand, the ramification divisor of
  $\alpha_{\kappa(X')}$ must be contained in $f_2^{-1}(D_1)_\red$, so
  (after this blowing up) that $\alpha_{\kappa(X')}$ has ramification
  divisor with strict normal crossings. We compose these two squares
  to arrive at the right-most square in the desired diagram.  We can
  also assume, by possibly taking a further prime-to-$\ell$ extension,
  that $R'$ has a primitive $\ell$th root of unity so that we may
  apply the results of Section~\ref{sec:ramification} to classes in
  the Brauer group.

  By Theorem~\ref{clever-splitsville} and
  Example~\ref{ex:pirutka_d_by_dsquared}, there exist rational functions
  $f_1, \dotsc, f_N$ in $\kappa(X')$ such that $\alpha_{L}$ is
  unramified, where $L =
  \kappa(X')(\sqrt[\ell]{f_1},\dotsc,\sqrt[\ell]{f_N})$.  (We can always
  choose $N=(d+1)^2$, and moreover, if there exists an $\ell$-Pirutka $n
  \times (d+1)$ matrix, we can take $N=n$.)  Let $Y$ be the normalization of
  $X'$ in $L$ and $g : Y \to X'$ the induced map.  We now apply
  Theorem~\ref{thm:gabber} again to arrive at an $\ell'$-alteration
  $$
  \xymatrix{
  Y'\ar[d]\ar[r]^h  &   Y\ar[d]\\
  \Spec R''\ar[r]    &   \Spec R'
  }
  $$
  where $Y'$ is regular and integral, and projective and
  quasi-semistable over $R''$.  Since $\alpha_{\kappa(Y)}$ is
  unramified, we have by Lemma \ref{lem:unram-climbs} that
  $\alpha_{\kappa(Y')}$ is unramified, whence $\alpha_{\kappa(Y')} \in
  \Br(Y')$ by Example \ref{ex:threefold-purity}.
\end{proof}

\section{Existence of twisted sheaves on strict normal crossings
  surfaces}\label{sec:snc}
\numberwithin{theorem}{section}
\numberwithin{equation}{section}

The purpose of this section is to show that period equals index for
Brauer classes on strict normal crossings surfaces. More precisely, if
$X$ is a snc surface over a separably closed or semi-finite field and
$\alpha\in\Hoh^2(X,\m_n)$ is a cohomology class with $n$ invertible on
$X$, we show in Proposition \ref{prop:snc} below that there is an
Azumaya algebra $A$ of degree $n$ on $X$ with cohomology class
$\alpha$ (equivalently, there exists a twisted sheaf of rank $n$ and trivial determinant).

We note first that Brauer groups of curves vanish in a high degree of
generality.

\begin{lemma}
  \label{lem:curves}
  Let $C$ be a curve over a field $k$ of characteristic $p \geq 0$.
  If $k$ is separably closed (resp. $k$ is semi-finite and $C$ is
  proper), then $\Br(C)[n]=0$ for $n$ prime to $p$ (resp. $\Br(C)=0$).
\end{lemma}

\begin{proof}
  If $k$ is separably closed, then this
  is~\cite{grothendieck-brauer-3}*{Corollaire~1.3}.  Thus, assume that
  $k$ is semi-finite and $C$ is proper over $k$.  Consider the Leray
  spectral sequence in \'etale cohomology 
  $$\Eoh_2^{st}=\Hoh^s(k,\Roh^t\pi_*\mathds{G}_{m,C})\Longrightarrow\Hoh^{s+t}(C,\mathds{G}_{m,C})$$
  for the structural morphism $\pi:C\rightarrow \Spec k$.  The only possible
  contributions to $\Hoh^2(C,\mathds{G}_{m,C})$ are
  \begin{align*}
      \Hoh^0(k,\Roh^2\pi_*\mathds{G}_{m,C})&,\\
      \Hoh^1(k,\Roh^1\pi_*\mathds{G}_{m,C})&\iso\Hoh^1(k,\ShPic_{C/k}),\\
      \Hoh^2(k,\Roh^0\pi_*\mathds{G}_{m,C})&\iso\Hoh^2(k,\Gm).
  \end{align*}
  The last vanishes because $k$ is semi-finite.  To analyze the second
  term, let $\tilde{C}\rightarrow C$ be the normalization of the
  largest reduced subscheme $C_{\mathrm{red}}$ of
  $C$. By~\cite{neronmodels}*{Corollary~9.2.11} (in fact, by reading the
  preceding Propositions carefully), there is an exact
  sequence
  $$0\rightarrow
  G\rightarrow\ShPic_{C/k}\rightarrow\ShPic_{\tilde{C}/k}\rightarrow
  0$$
  of \'etale sheaves over $k$, where $G$ is a geometrically connected
  commutative linear algebraic group. Moreover, since $\tilde{C}/k$ is
  smooth (since $k$ is perfect), there is
  (e.g.,~\cite{neronmodels}*{Propositions~9.2.3,~9.2.14}) an exact
  sequence
  $$0\rightarrow\ShPic_{\tilde{C}/k}^0\rightarrow\ShPic_{\tilde{C}/k}\rightarrow A\rightarrow
  0,$$
  where $\ShPic_{\tilde{C}/k}^0$ is an abelian variety and $A$ is an
  \'etale sheaf on $k$ with $A(k^s)\iso\ZZ^r$, where~$r$ is the number
  of irreducible components of $\tilde{C}_{k^s}$. In fact, $A$ becomes
  constant as soon as each component of $\tilde{C}$ acquires a
  rational point. In particular, $A_{k^s}$ is isomorphic to $\ZZ^r$,
  with the factors corresponding to the irreducible components, and
  hence permuted by the Galois group $\Gamma = \Gal(k^s/k)$. In
  particular, this means that $A_{k^s}$, as a Galois module, is of the
  form $\bigoplus_i \ZZ[\Gamma/H_i]$ for some open subgroups $H_i
  \subset \Gamma$. By Shapiro's Lemma, we then have that $\Hoh^1(\Spec k, A) =
  \Hoh^1(\Gamma, \bigoplus_i \ZZ[\Gamma/H_i]) = \prod_i \Hoh^1(H_i,
  \ZZ) = 0$.

  Note that since $k$ is semi-finite, we have that $\Hoh^1(\Spec
  k,G)=0$ for each geometrically integral commutative group scheme $G$
  over $k$. To see this, it suffices to show (since $G$ is
  commutative) that any $G$-torsor $T$ has a $0$-cycle of degree~$1$,
  for which it suffices to show that $T$ is trivial over the maximal
  prime-to-$\ell$ extension for each prime number $\ell$. But these
  fields are all pseudo-algebraicaly closed (by the definition of
  semi-finiteness), hence $T$ has a point over each of them, as $G$ is
  assumed to be geometrically integral.

  It follows by considering the above exact sequences that
  $\Hoh^1(k,\ShPic_{C/k})=0$. It remains to prove that
  $\Hoh^0(k,\Roh^2\pi_*\mathds{G}_{m,C})=0$. But, the stalk of
  $\Roh^2\pi_*\mathds{G}_{m,C}$ is isomorphic to $\Hoh^2(C_{k^s},\Gm)$, where
  $k^s$ is the separable closure of $k$. Since $k^s$ is algebraically closed (as
  $k$ is perfect), this group vanishes
  by~\cite{grothendieck-brauer-3}*{Corollaire~1.2}.
\end{proof}

\begin{remark}
There is also a proof that uses Tsen's theorem (resp.\ class field
theory) to treat the regular case and then deduces the general case by
deformation from points and a Moret-Bailly type formal gluing
argument, but we omit the details here.
\end{remark}

\begin{remark}
  The conclusion that $\Br(C)[n]=0$ for $n$ prime to $p$ cannot be improved
  to $\Br(C)=0$ without assuming that $k$ is algebraically closed. If $k$ is
  separably closed but not algebraically closed, then $\Br(k[x])$ is
  nonzero. This example appears already in Auslander and
  Goldman~\cite{auslander-goldman}*{Theorem~7.5}. Consider the Artin--Schreier extension $L$ of $k(x)$ defined by
  $y^p-y-x=0$. The ring $k[x,y]/(y^p-y-x)$ is easily seen to be smooth over
  $k$, and hence it is the integral closure of $k[x]$ in $L$. Since $k$ is
  not separably closed, there is an element $w\in k$ such that $w\notin
  k^p$. The algebra
  \begin{equation*}
    k[x]\langle y,z\rangle/(y^p-y-x,z^p-w,zy-yz-z)
  \end{equation*}
  defines an Azumaya algebra over $k[x]$. For more details, see Gille and
  Szamuely~\cite{gille-szamuely}*{Section~2.5}. This also explains why the
  full Brauer group is not $\AA^1$-homotopy invariant.
\end{remark}

The following lemma shows that the only obstruction to extending an
$\Xscr$-twisted locally free sheaf on a curve $C$ inside a surface $X$
is whether or not the determinant extends.  It is a direct
generalization to the twisted setting of~\cite{dejong}*{Lemma~5.2},
although the proof is slightly different owing to the fact that if the
$\m_n$-gerbe $\Xscr$ is nontrivial, then we cannot make use of an
$\Xscr$-twisted line bundle on $X$.

\begin{lemma}\label{lem:extension}
Let $C$ be a proper curve in a regular quasi-projective
$2$-dimensional scheme $X$ over a field $k$ of characteristic $p \geq
0$, and fix a $\m_n$-gerbe $\Xscr\rightarrow X$, where $n$ is prime
to $p$.  Suppose that $\Xscr$ has index $n$ and that the Brauer class
of $\Xscr$ vanishes on every proper curve in $X$, e.g., $k$ is
separably closed or semi-finite by Lemma~\ref{lem:curves}.

If $V$ is a locally free $\Xscr$-twisted sheaf on $C$ of
rank $n$ with $\det V=L|_C$, where $L\in\Pic(X)$, then, possibly
after taking a finite prime-to-$n$ extension of $k$, there exists a
locally free $\Xscr$-twisted sheaf $W$ on $X$ such that $W|_C\iso V$
and $\det(W)\iso L$.
\end{lemma}

\begin{proof}
Since $\Xscr$ has index $n$, there is a locally free $\Xscr$-twisted
sheaf $E$ of rank $n$, see
\cite{lieblich}*{Proposition~3.1.2.1(iii)}. Choose an ample line bundle
$\Oscr_X(1)$ on $X$. To prove the Lemma, we will use the following.

  \begin{claim}\label{claim:clam}
  There is an integer $m$, a proper curve $D\subset X$ with $\dim(D
  \cap C) = 0$ in the linear system $|L(mn)\otimes\det(E^\vee)|$, and
  an invertible $\Xscr$-twisted sheaf $M$ on $D$ such that there is an
  exact sequence of $\Xscr_C$-twisted sheaves
    \begin{equation*}
      0\rightarrow E(-m)|_C\rightarrow V\rightarrow M|_{C\cap D}\rightarrow 0.
    \end{equation*}
  \end{claim}

  We prove the lemma first assuming the claim. Let $$\gamma\in\Ext_C^1(M|_{C\cap
    D},E(-m)|_C)$$ be the corresponding (nonzero) extension class. Our goal is to lift
  $\gamma$ to $\Ext_X^1(M,E(-m))$. Given any open subset $U\subseteq X$ containing $C$, there is an
  exact sequence
  \begin{equation*}
    \Ext_U^1(M|_{U\cap D},E(-m)|_U)\rightarrow \Ext_C^1(M|_{C\cap D},E(-m)|_C)\rightarrow\Ext_U^2(M|_U,E(-m)(-C)|_U).
  \end{equation*}
  Now, $$\ShExt_U^0(M|_{U\cap D},E(-m)(-C)|_U)=0$$ since $M|_{U\cap D}$ is a torsion sheaf, while
  $$\ShExt_U^2(M|_{U\cap D},E(-m)(-C)|_U)=0$$ because
  $M|_{U\cap D}$ is a locally free sheaf on a curve in $U$ and hence has cohomological dimension~$1$.
  From the local-to-global ext spectral sequence, it follows that
  \begin{equation*}
    \Ext_U^2(M|_{U\cap D},E(-m)(-C)|_U)\iso\Hoh^1(U,\ShExt_U^1(M|_{U\cap D},E(-m)(-C)|_U)).
  \end{equation*}
  If we further choose $U$ to be such that $U\cap D$ is affine and
  $\dim(X\smallsetminus U)=0$, then this latter group vanishes
  since $\ShExt_U^1(M|_{U\cap D},E(-m)(-C)|_U)$ is supported on $D$ and so
  \begin{equation*}
    \Hoh^1(U,\ShExt_U^1(M|_{U\cap D},E(-m)(-C)|_U))\iso\Hoh^1(U\cap D,\ShExt_U^1(M|_{U\cap D},E(-m)(-C)|_U)|_D)=0,
  \end{equation*}
  by Serre's vanishing theorem for the cohomology of a quasi-coherent
  sheaf on an affine variety. It follows that $\gamma$ lifts to an
  extension
  \begin{equation*}
    0\rightarrow E(-m)|_U\rightarrow \tilde{V}\rightarrow M|_{U\cap D}\rightarrow 0
  \end{equation*}
  on $U$ such that $\tilde{V}|_C\iso V$. Let $W$ be $j_*\tilde{V}$, where $j:U\rightarrow
  X$ is the inclusion. Then, $W$ is reflexive since $X-U$ has codimension $2$.
  By construction it restricts to $V$ on $C$, and since $S$ is regular and $2$-dimensional, $W$ is locally free.

  The determinant of $W$ is
  $$\det(E(-m))\otimes\det(M)\iso\det(E)(-mn)\otimes\det(M).$$ Since $M$ is a
  locally free $\Xscr$-twisted sheaf of rank $1$ on $D$,
  $\det(M)\iso\Oscr_X(D)$ (see~\cite{lieblich}*{Proposition~A.5}).
  But $D$ was chosen to be in the class of the linear system associated to
  $L(mn)\otimes\det(E^\vee)$. It follows immediately that $\det(W)\iso L$, as
  desired.

  Now we prove Claim \ref{claim:clam}. For sufficiently large $m$, a
  general map in $\Hom(E|_C,V(m))$ is injective and has a cokernel
  isomorphic to the pushforward of an invertible $\Xscr$-twisted line bundle
  on a general member of the linear system $|N|_C|$, where
  $N|_C=\det(V(m))\otimes\det(E|_C^{\vee})$. (We suppress the fact
  that $N$ depends on $m$ in the notation.)  This follows
  from~\cite{lieblich-moduli}*{Corollary~3.2.4.21}; if $k$ is finite,
  to use the required Bertini theorem we can take arbitrarily large
  finite prime-to-$n$ extensions of $k$ to ensure the existence of
  rational points avoiding the ``forbidden cone'' (as any open subset
  of affine space over an infinite field contains rational points).
  By assumption, the line bundle $N|_C$ is the restriction of the line
  bundle $N=L(mn)\otimes\det(E^\vee)$ on $X$.  For sufficiently large
  $m$, $N$ is ample, and a general member of $|N|$ restricts to a
  general member of $|N|_C|$. We let $D$ be a general regular member
  of $|N|$ such that $D\cap C$ is the support of an injective map
  $E|_C\rightarrow V(m)$ with cokernel the pushforward of an invertible
  $\Xscr$-twisted sheaf on $D\cap C$. By
  hypothesis, the Brauer class of $\Xscr$ vanishes on $D$, so there is
  an invertible $\Xscr$-twisted sheaf $M$ on $D$. This proves the
  claim.
\end{proof}

Before getting to the main result, we need to extend a standard result
about elementary transformations to the case of a strict normal
crossings scheme.  The case of a regular scheme is handled in
\cite{lieblich}*{Corollary~A.7}.

\begin{definition}
Suppose that $Z$ is an algebraic stack and that $i:W\subset Z$ is a closed substack.
Furthermore, suppose we are given a quasi-coherent sheaf $F$ on $Z$ and a
quotient $q:F|_W\to Q$. The \df{elementary transform} of $F$ along $q$
is defined to be the kernel of the morphism $F\to i_\ast Q$ induced by
the adjunction map and $q$.
\end{definition}

\begin{lemma}\label{lem:elementary}
Let $X$ be a scheme and $C\subset X$ an effective Cartier divisor with
connected component decomposition $C=\sqcup_i C_i$. Suppose that
$\pi:\Xscr\to X$ is a $\m_n$-gerbe and that $E$ is a locally free
$\Xscr$-twisted sheaf. If $q:E|_C\to F$ is a surjection to a
locally free $\Xscr$-twisted sheaf $F$ supported on $C$, then the determinant of the elementary
transform of $E$ along $q$ is isomorphic to
  \[\det(E)\otimes\Oscr_X(-\sum_im_iC_i),\]
where $m_i$ is the rank of $F|_{C_i}$.
\end{lemma}

\begin{proof}
In order for the determinant to be well-defined, we need to check that
the subsheaf $G :=\ker(E \to i_*F)$ is perfect when viewed as a
complex of $\Xscr$-twisted $\Oscr_{\Xscr}$-modules with quasi-coherent
cohomology sheaves. In fact, $G$ is locally free. To see this, it
suffices to work smooth-locally and prove the following: let $Z$ be a
scheme, $i:W\hookrightarrow Z$ an effective Cartier divisor, and
$E$ a locally free sheaf on $Z$. Given a locally free sheaf
$F$ on $W$, the kernel of any surjection $E\twoheadrightarrow
i_\ast F$ is locally free. This in turn reduces to the local case,
so we may assume that $Z=\Spec A$, that $W$ is cut out by a single
regular element $a\in A$, and that $E$ and $F$ are free on $Z$
and $W$, respectively. Since $\Oscr_W$ has projective dimension $1$
over $A$, so does $F$. But since $E$ has projective dimension
$0$, it follows that the kernel of any such surjection must also be
projective, hence locally free, as desired.

To prove the Lemma, it is enough to verify that
$\det(i_*F)\iso\Oscr_X(\sum_i m_i C_i)$.  Assume henceforth that $C$
is connected, and hence that $F$ has constant rank, say $m$,
everywhere on~$C$. Pulling back to the Severi--Brauer scheme $P\to X$
associated to the Azumaya algebra $\pi_\ast\ShEnd(E)$ (so that
$\Xscr|_P$ has trivial Brauer class) and using the fact that
$\Pic(X)\to\Pic(P)$ is injective, we are immediately reduced to the
analogous statement for trivial Brauer classes.  Let $L$ be an
invertible $\Xscr$-twisted sheaf.  The classical theory of
determinants tells us that $\det(i_\ast F\tensor
L^\vee)\cong\Oscr(mC)$. But the rank of $i_\ast F$ is $0$, so this
also computes $\det(i_\ast F)$, as desired.
\end{proof}

\begin{proposition}\label{prop:snc}
Let $X$ be a quasi-projective geometrically connected snc surface over
a field $k$ of characteristic $p\geq 0$, and let $\Xscr\rightarrow X$
be a $\m_n$-gerbe, where $n$ is prime to $p$.  Suppose that $\Xscr$
has index $n$ on each irreducible component of $X$ and that the Brauer
class of $\Xscr$ vanishes on each closed subscheme of $X$ of dimension
at most $1$.  (This later condition holds when $k$ is separably closed
or $X$ is proper and $k$ is semi-finite by Lemma~\ref{lem:curves}.)
Then there exists a locally free $\Xscr$-twisted sheaf of rank $n$ and
trivial determinant on $X$.
\end{proposition}

\begin{proof}

First, we show that if there exists an $\Xscr$-twisted sheaf $F$ of
rank $n$ on $X$, then $F$ can be chosen to have trivial determinant.
Indeed, for $m \gg 0$, we can assume that $\det(F(m))=\Oscr_{Y}(D)$,
where $D$ is an effective Cartier divisor on $X$. By choosing an
invertible $\Xscr$-twisted sheaf on $D$ (which is possible by the
assumption that the Brauer class of $\Xscr$ vanishes on curves), we
can find an invertible quotient $Q$ of $F(m)|_D$. By
Lemma~\ref{lem:elementary}, the elementary transform of $F(m)$ along
$Q$ has trivial determinant.  Thus, we have constructed a locally free
$\Xscr$-twisted sheaf of rank $n$ on $Y$ with trivial determinant.

  Now we proceed by induction on the number of irreducible components
  of $X$.  If $X$ is irreducible (hence regular), then the existence
  of a locally free $\Xscr$-twisted sheaf $F$ of rank $n$ on $X$
  follows from the fact that $\Xscr$ has index $n$ and the existence
  of Azumaya maximal orders over a regular surface.  By the above, we
  can choose $F$ to have trivial determinant.

  In general, let $X=X_1\cup\cdots\cup X_r$ be the decomposition of
  $X$ into its irreducible components.  Assume that there exists a
  locally free $\Xscr$-twisted sheaf $F$ of rank $n$ on
  $Y=X_1\cup\cdots\cup X_{r-1}$.  Let $C=Y\cap X_{r}$. By the above,
  we can choose $F$ with trivial determinant.  Consequently the
  restriction of $F$ to $C$ has trivial determinant, which coincides
  with the restriction $\Oscr_{X_{r}}|_C$ of the trivial line bundle
  from $X_r$.  Hence by Lemma~\ref{lem:extension}, there exists a
  locally free $\Xscr$-twisted sheaf $F_r$ on $X_r$ such that $F|_C$
  is isomorphic to $F_r|_C$.  Let $E$ be the fiber product of $F$ and
  $F_r$ over their restrictions to $C$ (via the chosen isomorphism) in
  the abelian category of $\Xscr$-twisted sheaves on $X$.  By applying
  \cite{milnor-kbook}*{Theorem~2.1} \'etale-locally, we see that $E$
  is locally free of rank $n$ on $X$, as desired.  By the above, we
  can choose $E$ with trivial determinant.

  By induction, we produce the desired locally free $\Xscr$-twisted
  sheaf on $X$.
\end{proof}

The following corollary, which in particular asserts that index
equals period, may be found in \cite{lieblich}*{Corollary~4.2.2.4}
in the case when $X$ is smooth over a separably closed field.

\begin{corollary}\label{cor:reduced}
  Under the hypotheses of Proposition~\ref{prop:snc}, the map
  \begin{equation*}
    \Hoh^1(X,\PGL_n)\rightarrow\Hoh^2(X,\m_n)
  \end{equation*}
  is surjective.
\end{corollary}

\begin{proof}
  Given a $\m_n$-gerbe $\Xscr\rightarrow X$, the proposition produces
  a locally free $\Xscr$-twisted sheaf $V$ of rank $n$. The determinant of $V$ differs
  from $[\Xscr]$ by a class of $\Pic(X)/n\Pic(X)$. Performing an elementary transformation
  along a suitable effective Cartier divisor corrects the determinant, by
  Lemma~\ref{lem:elementary}.
\end{proof}

\section{Deformation theory of perfect twisted
  sheaves}\label{sec:deformations}
\numberwithin{theorem}{subsection}

\subsection{Generalities}
\label{sec:generalities-defmns}

The material in this section is similar
to~\cite{lieblich-moduli}*{Section~2.2.3}, except our infinitesimal
deformations of the ambient scheme are not assumed to be flat over a
base. We review the theory in this case; there are no essential
differences.

Let $i:X_0\hookrightarrow X$ be a closed subscheme of a
quasi-separated noetherian scheme $X$ defined by a square-zero sheaf
of ideals $I$ of $\Oscr_X$. Let $\pi : \Xscr\to X$ be a
$\m_\ell$-gerbe, write $\Xscr_0 = \Xscr \times_X X_0$ and $\pi_0 :
\Xscr_0 \to X_0$ for the restriction of $\pi$, and write $\iota :
\Xscr_0 \to \Xscr$ for the induced closed immersion. We write
$\Drm_{\qc}^{(1)}(\Xscr)$ for the derived category of $\Xscr$-twisted
sheaves with quasi-coherent cohomology. Let $F_0$ be an object in
$\Drm_{\qc}^{(1)}(\Xscr_0)$.

\begin{definition}
  A \df{deformation} of $F_0$ to $\Xscr$ consists of a complex
  $F$ in $\Drm_{\qc}^{(1)}(\Xscr)$ and a quasi-isomorphism
  $\Oscr_{\Xscr_0}\otimes_{\Oscr_\Xscr}^{\Lrm}F\we F_0$.
\end{definition}

For convenience, we write $I \otimes^\Lrm F$ for the complex of
$\Xscr$-twisted sheaves $\pi^\ast I \otimes_{\Oscr_{\Xscr}}^\Lrm F$
and $I \otimes^\Lrm F_0$ for the complex of $\Xscr_0$-twisted sheaves
$\pi_0^\ast i^\ast I \otimes_{\Oscr_{\Xscr_0}}^\Lrm F_0$.

\begin{lemma}
  If $F_0$ is perfect and $F$ is a deformation of $F_0$ to $\Xscr$,
  then $F$ is perfect.
\end{lemma}

\begin{proof}
Note that there is a distinguished triangle $I
\otimes^{\mathrm{L}}F \rightarrow F \rightarrow \iota_*
F_0$ in $\Drm_{\qc}^{(1)}(\Xscr)$.  If we prove that $F$ has finite
Tor-amplitude, then it will follow
from~\cite{thomason-trobaugh}*{Theorem~2.5.5} that
$I\otimes^{\mathrm{L}}F$ is quasi-isomorphic to
\[
\iota_*\left(I\otimes^{\mathrm{L}}\Lrm \iota^* F\right)\we
    \iota_*(I\otimes F_0).
\]
On a quasi-separated noetherian scheme $Y$ the perfect complexes of
$\Oscr_Y$-modules are precisely those complexes which have coherent
cohomology sheaves and which moreover have bounded
Tor-amplitude (see~\cite{thomason-trobaugh}*{Example~2.2.8 and
Proposition~2.2.12}). Thus, choosing an \'etale covering of $X$
splitting $\Xscr\to X$, we see that the same holds for complexes on
$\Xscr$.  Since $\iota_*F_0$ and $\iota_*(I\otimes F_0)$ have
coherent cohomology sheaves, it follows that if we show that $F$ has
finite Tor-amplitude, the lemma will follow.

    Recall that a complex $F$ of $\Drm_{\qc}^{(1)}(\Xscr)$ has Tor-amplitude
    contained in an interval with integer endpoints $[a,b]$ if and only if
    $\ShTor_n^{\Oscr_\Xscr}(G,F)=0$ for all $\Oscr_\Xscr$-modules
    $G$ and all $n\notin[a,b]$. Suppose that $F_0$ has Tor-amplitude
    contained in $[a,b]$. Suppose that $n\in\ZZ$, and suppose that
    $\ShTor_n^{\Oscr_\Xscr}(G,F)$ is not zero. Then, there is a closed
    point $x$ of $X$ such that $\ShTor_n^{\Oscr_\Xscr}(k(x),F)$ is not zero.
    But
    $$F\otimes_{\Oscr_\Xscr}^{\Lrm}k(x)\we F
    \otimes_{\Oscr_\Xscr}^{\Lrm}(\Oscr_{\Xscr_0}\otimes_{\Oscr_{\Xscr_0}}^{\Lrm}k(x))\we
    F_0\otimes_{\Oscr_{\Xscr_0}}^{\Lrm}k(x).$$
    Hence,
    $$\ShTor_n^{\Oscr_\Xscr}(k(x),F)\iso\ShTor_n^{\Oscr_{\Xscr_0}}(k(x),F_0),$$
    which implies that $n\in[a,b]$, as desired.
\end{proof}

\begin{definition}
  Recall that there is an essentially unique \emph{determinant\/} functor
  $$\det:\Perf(\Xscr)\to\Pic(\Xscr)$$
  that associates to each perfect complex of $\Xscr$-twisted sheaves
  an invertible sheaf. Given a perfect complex $F\in\Perf(\Xscr)$ and
  an fppf cover $Y\to\Xscr$ over which $F\we C^\bullet$, a finite
  complex of locally free sheaves, the determinant is computed
  as $$\det(F)=\bigotimes_{n\in\ZZ}\det(C^n)^{(-1)^n}.$$ (See
  \cite{MR0437541}*{Theorem~2},
  \cite{lieblich-moduli}*{Definition~2.2.4.1}.)
\end{definition}

\begin{definition}
  Suppose that $F$ has $\det(F_0)\iso\Oscr_{\Xscr_0}$, and fix one such
  trivialization. An \df{equideterminantal deformation} of $F_0$ is a
  deformation $F$ as above together with a deformation of the
  trivialization of the determinant.
\end{definition}

The next proposition is well-known to experts and follows immediately
from the techniques in~\cite{illusie}, cf.\
\cite{lieblich-moduli}*{Proposition~2.2.4.9}.

\begin{proposition}\label{prop:deformations}
Let $X_0\hookrightarrow X$ be a closed subscheme of a quasi-separated
noetherian scheme $X$ defined by a square-zero sheaf of ideals $I$ of
$\Oscr_X$. Fix a $\m_\ell$-gerbe $\Xscr\to X$, and let
$\Xscr_0=\Xscr\times_X X_0$. Suppose $F_0$ is a perfect complex of
$\Xscr_0$-twisted sheaves. Then the obstruction to
the existence of an equideterminantal deformation of $F_0$ to $\Xscr$
lies in
$$\ShHoh^2(X_0,I\otimes^{\Lrm}\sRShEnd(F_0)),$$ where
$\sRShEnd(F_0)$ denotes the trace zero part of the complex of
endomorphism sheaves.
\end{proposition}

\begin{remark}
    In Proposition~\ref{prop:deformations}, when the rank of $F_0$ is invertible on
$X_0$, we can compute
$\ShHoh^2(X_0,I\otimes^{\Lrm}\sRShEnd(F_0))$
as $\Ext_{\Xscr_0}^2(F_0, I\otimes F_0)_0$, the kernel of the trace map
  $$\Ext_{\Xscr_0}^2(F_0,I\otimes F_0)\to\Hoh^2(\Xscr_0,I)$$
  on cohomology. Indeed, in this case the trace map
  $\RShEnd(F_0)\to\Oscr_{\Xscr_0}$ splits because the composition
  $$\Oscr_{\Xscr_0}\to \RShEnd(F_0)\to\Oscr_{\Xscr_0}$$
  is multiplication by the rank of $F_0$ (see \cite{lieblich-moduli}*{Lemma 2.2.4.5}).
\end{remark}

\subsection{Fracking}\label{sec:fracking}

In the next two sections we describe a standard trick in deformation
theory that kills obstructions in dimension $2$. Because the general
local tool we use roughly corresponds to ``punching holes'' in a
sheaf, we call this \textit{fracking}.

Let $X$ be a locally noetherian scheme, $\pi:\Xscr\to X$ a $\G$-gerbe for some closed
subgroup $\G\subset\G_m$, and $F$ a locally free $\Xscr$-twisted sheaf
of finite rank. Let $i:\Spec K\to X$ denote a closed immersion whose
image is a regular point $x$ of $X$, where $K$ is a field. We will write
$\Xscr_0=\Xscr\times_X\Spec K$, we will let $\pi_0:\Xscr_0\to\Spec K$
denote the restriction of $\pi$, and we will let $\iota:\Xscr_0\to\Xscr$ denote the natural closed immersion.

Given two $\Oscr_X$-modules $M$ and $N$, there is a trace map
\[
\Hom(M\tensor F, N\tensor F)\to\Hom(M,N)
\]
induced by the isomorphism
\[
\ShHom(M\tensor F,N\tensor F)\iso\ShHom(M,N)\tensor\ShHom(F,F)
\]
and the usual trace map. We will let $\Hom(M\tensor F,N\tensor
F)_0\subset\Hom(M\tensor F,N\tensor F)$ denote the kernel of this
trace map.

\begin{lemma}[Fracking Lemma]\label{lem:frak}
  With the above notation, suppose that
  \begin{enumerate}
  \item $\dim\Oscr_{X,x}\geq 2$;
  \item the rank of $F$ is prime to the characteristic of $K$;
  \item $M$ and $N$ are invertible in a neighborhood of $x$;
  \item the class of $\Xscr_0$ in $\Br(K)$ is trivial;
  \item $f$ is an element of $\Hom(M\tensor F,N\tensor F)_0$ whose image in
  $\Hom(\iota^\ast(M\tensor F),\iota^\ast(N\tensor F))_0$ is non-zero;
  \end{enumerate}
  Then there exists a locally free $\Xscr_0$-twisted sheaf $Q$ of rank
  $\operatorname{rk}(F)-1$ and a surjection $\iota^\ast F \to Q$ such
  that, writing $G$ for the kernel of the adjoint map $F\to \iota_\ast
  Q$, the endomorphism $f$ is not in the image of the natural
  inclusion
\[
\rho:\Hom(M\tensor G,N\tensor G)_0\rightarrow\Hom(M\tensor F,N\tensor F)_0
\]
induced by the canonical isomorphism $G^{\vee\vee}\to F$.
\end{lemma}
\begin{proof}
  Since $\Xscr_0$ has trivial Brauer class, we can choose an
invertible $\Xscr_0$-twisted sheaf~$\Lambda$; the sheaf $\Lambda$ is
unique up to non-unique isomorphism. Define two functors
$$\tau_\dagger:\Coh^{(1)}(\Xscr_0)\to\Coh(\Spec K)$$
and
$$\tau^\dagger:\Coh(\Spec K)\to\Coh^{(1)}(\Xscr_0)$$
by the formulas $\tau_\dagger(A)=(\pi_0)_\ast(A\tensor\Lambda^\vee)$
and $\tau^\dagger(B)=\Lambda\tensor\pi_0^\ast B$. It follows from
these formulas and the basic theory of twisted sheaves that
$\tau^\dagger$ and $\tau_\dagger$ are essentially inverse
equivalences. Given a quasi-coherent sheaf $M$ on $X$ and an object
$A\in\Coh^{(1)}(\Xscr_0)$,
there is a natural
isomorphism \numberwithin{equation}{theorem}
\begin{equation}\label{eq:eq}
  \tau_\dagger(\iota^\ast\pi^\ast M\tensor A)\iso i^\ast M\tensor\tau_\dagger(A).
\end{equation}

Write $\overline F=\tau_\dagger(\iota^\ast F)$. By equation \eqref{eq:eq} and the
fact that $i^\ast M$ and $i^\ast N$ are $1$-dimensional $K$-vector spaces, we can
transport $\iota^\ast f$ to a non-trivial traceless
homomorphism
$$\overline f:i^\ast M\tensor \overline F\to
i^\ast N\tensor \overline F$$
of $K$-vector spaces of the same (finite) dimension, see
\cite{lieblich}*{Theorem~3.1.1.11}.  Because the nonzero $\overline f$
has trace zero and all nonzero scalar matrices have nonzero trace (by
the assumption that $F$ has rank prime to the characteristic of $K$),
there is a line $\overline{L}$ in $\overline{F}$ such that $\overline
f$ does not preserve $\overline{L}$ (i.e., $\overline f(i^\ast
M\tensor \overline{L})$ is not contained in $i^\ast
N\tensor\overline{L}$).

Let $Q$ be the $\Xscr_0$-twisted sheaf
$\pi^\dagger\overline{F}/\pi^\dagger\overline{L}$ and $\sigma: F\to
\iota_\ast Q$ the adjoint of the natural surjection. Write $G$ for the
kernel of $\sigma$ and $\gamma:G\to F$ for the inclusion. Note that
the canonical map $G^{\vee\vee}\to F$ is an isomorphism. (Here we use
that $F$ is locally free, hence $G$ is locally free away from $x$, and
that $x$ itself is a regular point.)  Since $M$ and $N$ are invertible
near $x$, it follows that there is an induced canonical inclusion
$$\rho:\Hom(M\tensor G,N\tensor G)_0\hookrightarrow\Hom(M\tensor F,N\tensor F)_0.$$
(The one subtle point is the
preservation of the trace zero
condition. This follows since $M$ and $N$ are invertible near $x$ and
$F$ is locally free, so the traceless condition can be detected on the
punctured neighborhood of $x$.)

This canonical inclusion has the property that for any
$s\in\Hom(M\tensor G,N\tensor G)$ we have a commuting diagram
$$
\xymatrix{
M \tensor G \ar[r]^s \ar[d]_{\id_M \tensor \,\gamma} & N \tensor G \ar[d]^{\id_N \tensor \,\gamma}\\ 
M \tensor F \ar[r]_{\rho(s)} & N \tensor F. 
}
$$
It follows that the image of $\rho$ lies in the subgroup $B$ of
$\Hom(M\tensor F,N\tensor F)_0$ of those trace zero
endomorphisms whose restrictions to the fiber over $x$ map the flag
$i^\ast M\tensor\pi^\dagger\overline{L}\subseteq i^\ast M\tensor\pi^\dagger\overline{F}$
into the flag $i^\ast N\tensor\pi^\dagger\overline{L}\subseteq i^\ast
N\tensor\pi^\dagger\overline{F}$. On the other
  hand, there is an exact sequence
  $$0\rightarrow
  B\rightarrow\Hom(M\tensor F,N\tensor F)_0\rightarrow\Hom\left(M\tensor
  \pi^\dagger\overline{L},N\tensor\left(\pi^\dagger\overline{F}/\pi^\dagger\overline{L}\right)\right).$$
  Since the endomorphism $f$ we started with is nonzero on the right,
  it is not contained in $B$, and hence is not in the image of $\rho$, as
  desired.
\end{proof}

\subsection{Removing global obstructions by fracking}
\label{sec:remove-frack}

In this section, we explain how to use Lemma \ref{lem:frak} to produce
unobstructed twisted subsheaves in dimension 2.

\begin{situation}\label{sit:frack}
Suppose that $X$ is a proper Gorenstein surface over a field $k$ that is
either semi-finite or separably closed, $\Xscr\to X$ is a
$\m_\ell$-gerbe, $F$ is a perfect coherent $\Xscr$-twisted sheaf
whose rank is invertible in $k$, and $M$ is a coherent sheaf on $X$
that is the pushforward of an invertible sheaf on a closed subscheme
$X'$ of $X$ that contains a nonempty open subscheme $U\subset X$ (for
example, $M$ could be an invertible sheaf on a component of $X$).
\end{situation}
In Situation \ref{sit:frack}, there are two trace maps
$$\Hom(M\tensor F,\omega_X\tensor F)\to\Gamma(X,M)$$
and
$$\Ext_{\Xscr}^2(F, M\tensor F)\to\Gamma(X,M).$$
Via Serre duality, there is an isomorphism of trace zero subspaces
$$\Ext_{\Xscr}^2(F,M\tensor F)_0=\Hom(M\tensor F,\omega_X\tensor F)_0^\vee.$$

\begin{proposition}\label{prop:frack-attack}
In Situation \ref{sit:frack}, there is an $\Xscr$-twisted subsheaf
$G\subset F$ such that
  \begin{enumerate}
  \item $F/G$ is supported at finitely many regular closed points of $X$ whose
  residue fields are separable extensions of $k$, and
  \item $\Ext_{\Xscr}^2(G,M\tensor G)_0=0$.
  \end{enumerate}
\end{proposition}
\begin{proof}
Let $f:M\otimes F\rightarrow\omega_X\otimes F$ be a nonzero element
of the $k$-vector space of trace zero homomorphisms $\Hom(M\otimes
F,\omega_X\otimes F)_0$.  Choose a regular closed point $x$ of $X$
with separable residue field such that $f$ is nonzero at $x$. This is
possible since $M$ is invertible on $X'$ and $F$ is locally free on a
dense open of $X'$. Since $M\otimes F$ and $\omega_X\otimes F$ are
isomorphic at $x$, we can apply Lemma~\ref{lem:frak} at $x$ (using the
fact that $\kappa(x)$ has trivial Brauer group and the fact that the fibers of
$\omega_X$ and
$M$ are one-dimensional) to obtain a subsheaf $G\subseteq F$ such that
\begin{enumerate}
\item $G$ is a perfect sheaf with reflexive hull $F$,
\item   the map $$\Hom(M\otimes G,\omega_X\otimes
  G)_0\rightarrow\Hom(M\otimes F,\omega_X\otimes F)_0$$
  induced by passing to reflexive hulls are injective, but
\item  $f$ is not in the image.
\end{enumerate}
Since (2) and (3) imply that that the dimension of $\Hom(M\otimes
G,\omega_X\otimes G)_0$ is strictly smaller than that of
$\Hom(M_i\otimes F,\omega_X\otimes F)$, we can, possibly after
repeating the process finitely many times, find $G$ such that
$$\Ext_{\Xscr}^2(G,M\tensor G)_0^\vee=0,$$
as desired.
\end{proof}

\section{Proofs of the main results}\label{sec:proof}

In this section we provide the proofs of Theorems~\ref{thm:main}
and~\ref{thm:1}.  As a standard reduction, we may, first of all,
assume that $\per(\alpha)=\ell$ is a prime distinct from the residue
characteristics of $X$, see
\cite{artin:Brauer-Severi}*{Proof~of~Theorem~6.2}. Since the index is
preserved under taking prime-to-$\ell$ extensions, we may adjoin a
primitive $\ell$th root, if necessary, so that we can apply the
results of Section~\ref{sec:ramification} to classes in the Brauer
group.  By Proposition \ref{prop:diagram} and the results of
Section~\ref{sec:some-pirutka-matr} (specifically
Example~\ref{3x3_clever} for $\ell >3$ and
Example~\ref{4x3_cheeseburger} for $\ell \mid 6$), the proofs of
Theorem~\ref{thm:main} and Theorem~\ref{thm:1} both reduce to proving
Theorem~\ref{thm:1} under the additional hypothesis that $X\to\Spec R$
is quasi-semistable. Let $\Xscr\rightarrow X$ be the $\m_\ell$-gerbe
associated to $\alpha$.

Let $\pi$ be a uniformizing parameter of $R$, so that $(\pi)$ denotes
the sheaf of ideals in $\Oscr_X$ that cuts out the special fiber
$X_0$, and let $I\supseteq(\pi)$ be the sheaf of ideals in $\Oscr_X$
that cuts out the \textit{reduced\/} special fiber $X_{0,\red}\subseteq
X_0$. Write $\Xscr_{0,\red}\to X_{0,\red}$ for the restriction $\Xscr\times_X
X_{0,\red}$. By Proposition~\ref{prop:snc}, there exists an $\Xscr_{0,\red}$-twisted
sheaf $F$ of rank $\ell$ with trivial determinant.  To finish
the proof, it suffices to find a perfect twisted subsheaf $G\subset F$
such that $\operatorname{rank}(G)=\operatorname{rank}(F)$ such that
$G$ deforms to an $\Xscr$-twisted sheaf over the formal scheme
$\widehat X$. Indeed, by the Grothendieck Existence
Theorem~\cite{ega3_1}*{Th\'eor\`eme~5.1.4}, any such formal
deformation algebraizes to yield an $\Xscr_{\widehat{R}}$-twisted sheaf of rank
$\ell$ on $X_{\widehat{R}}$, the pullback of $X\rightarrow\Spec R$ to the
completion $\widehat{R}$ of $R$. By Artin approximation, there is thus
a coherent $\Xscr$-twisted sheaf $V$ of rank $\ell$. By
\cite{lieblich}*{Proposition 3.1.2.1}, we have that
$\ind(\alpha_{\kappa(X)})$ divides $\ell$, as desired.

The rest of this section is devoted to producing the desired formal
deformation. We will do this by analyzing the formal local structure
of $X$ near $X_{0,\red}$ and then applying Lemma~\ref{lem:frak} to eliminate
obstructions to deforming across infinitesimal neighborhoods of $X_{0,\red}$.

Given two sheaves of ideals $I_1$ and $I_2$ on a scheme $Y$, define
$$I_1\starplus I_2=(I_1I_2:I_1\cap I_2).$$
If $f\in\Gamma(Y,\Oscr_Y)$ is an everywhere regular section, then we
have $(f I_1:I_2)=f(I_1:I_2)$ and $(f I_1:f I_2)=(I_1:I_2)$, so that $fI_1\starplus f I_2=f(I_2\starplus I_2)$. If $Y$ is the spectrum of a UFD, then for
any two sections $f, g\in\Gamma(Y,\Oscr_Y)$, we have
\begin{equation}\label{eq:fg}
  (f)\starplus(g)=(\gcd(f,g))
\end{equation}
Since this can be checked locally, it also follows that~\eqref{eq:fg}
holds in any locally factorial scheme.

Since $I/(\pi)$ is nilpotent, there is a least $m$ such that $I^m\subseteq(\pi)$.
Given $1\leq a\leq m$ and $b\geq 0$, let
$J_{a,b}=I^a(\pi^b)\starplus(\pi^{b+1})$.
The ideals $J_{a,b}$ have the following properties.
\begin{enumerate}
    \item   $J_{a+1,b}\subseteq J_{a,b}$ for $1\leq a\leq m-1$ and
        $J_{1,b+1}\subseteq J_{m,b}$.
    \item $J_{m,b}=(\pi^{b+1})$.
        Indeed, $I^m(\pi^b)\subseteq(\pi^{b+1})$, so that $(\pi^{b+1})\subseteq
        I^m(\pi^b)\starplus(\pi^{b+1})$. Since $X$ is regular, $I$ is locally
        principal, so that the inclusion $(\pi^{b+1})\subseteq J_{m,b}$ is
        locally an equality and hence an equality.
    \item $J_{a,b}/J_{a+1,b}\iso J_{a,0}/J_{a+1,0}$ for $1\leq a\leq (m-1)$
      and $J_{m,b}/J_{1,b+1}\iso J_{m,0}/J_{1,1}\iso\Oscr_X/I$. This also follows
      from the fact that $\pi$ is a regular section of $\Oscr_X$.
\end{enumerate}
Consider the filtration
    \begin{align*}
        I=J_{1,0} &\supset J_{2,0} \supset
        \ldots\supset J_{m-1,0}\supset (\pi)=J_{m,0}\supset \numberthis\label{eq:sequence}\\
      J_{1,1}&\supset J_{2,1}\supset
        \ldots\supset J_{m-1,1}\supset (\pi^2)=J_{m,1}\supset \\
        J_{1,2}&\supset J_{2,2}\supset
        \ldots\supset J_{m-1,2}\supset (\pi^3)=J_{m,2}\supset\\
        \ldots.
    \end{align*}
By the above calculations, there are only finitely many $\Oscr_{X}$-modules
appearing in the list of successive quotients in this filtration. By our choice of $m$, all of the successive quotients are nonzero. Moreover, multiplication by $I$ kills any of these $\Oscr_X$-modules, so we can
view them as $\Oscr_{X_0,\red}$-modules.

\begin{claim*}
Each successive quotient in the filtration defined
in~\eqref{eq:sequence} is locally free of rank $1$ on its support,
which consists of the union of a set of components of $X_{0,\red}$.
\end{claim*}

The claim is immediate for
$J_{m,b}/J_{1,b+1}\iso\Oscr_X/I=\Oscr_{X_0,\red}$.  For $1\leq a<m$, we
verify the claim \'etale locally, where we can appeal to the \'etale
local structure \eqref{eq:localform} of $X$. Thus, we may assume that
our regular scheme is $X=\Spec R[t_1,\ldots,t_n]/(t_1^{a_1}\cdots
t_r^{a_r}-\pi)$. In this case, $I=(t_1\cdots t_r)$ and
$(\pi)=(t_1^{a_1}\cdots t_r^{a_r})$. Using~\eqref{eq:fg}, we find
that $$J_{a,0}=(t_1^{\min(a,a_1)}\cdots t_r^{\min(a,a_r)}),$$ and
hence the quotient $J_{a,0}/J_{a+1,0}$ is isomorphic
to $$\Oscr_X/(t_1^{\epsilon(1,a)}\cdots t_r^{\epsilon(r,a)}),$$
where $$\epsilon(i,a)=\begin{cases}
0   &   \text{if $a_i\leq a$,}\\
1 & \text{if $a_i>a$.}
\end{cases}$$
Note that, by our choice of $m$, for any $1\leq a<m$, some
$\epsilon(i,a)$ is nonzero. It follows that the successive quotient
is (\'etale locally) isomorphic to the structure sheaf of some
collection of components of the reduced special fiber, proving the
claim.

For notational simplicity, set $M_0=\Oscr_{X_0,\red}$ and $M_i=J_{i,0}/J_{i+1,0}$ for $1\leq i<m$.
We claim that there is a perfect $\Xscr_{0,\red}$-twisted subsheaf $G\subseteq F$ such that
$F/G$ is supported in dimension $0$ and
$$
\Ext_{\Xscr_{0,\red}}^2(G,M_i\otimes G)_0=0
$$ for $0\leq i<m$.
If this is so then the obstruction of
Proposition~\ref{prop:deformations} to deforming any such $G$ (with
trivial determinant) through
the filtration \eqref{eq:sequence} vanishes, giving the desired formal
deformation. But this follows from Proposition
\ref{prop:frack-attack} applied in sequence to $M_1,\ldots,M_{m-1}$.

\begin{remark}
In the proof, we may have to take a torsion free subsheaf of $F$ in
order to remove obstructions to deforming off of $X_{0,\red}$. (We need $G$
to be perfect so that taking its determinant makes sense, and we work
with the equideterminantal deformations in order to kill obstruction
spaces using Proposition \ref{prop:frack-attack}.) In that case,
the resulting $\Xscr$-twisted sheaf may not be locally free. Moreover, a
reflexive sheaf on a regular threefold need not be locally free,
though it will have torsion free fibers over $R$.  Algebraically
speaking, this process may yield a maximal order in the division
algebra corresponding to $\alpha$ that is not locally free (see~\cite{ant-will}).
\end{remark}

\bibliographystyle{plain}
\bibliography{squarepaper}{}

\end{document}